\documentclass[a4paper,11pt]{amsart}
\usepackage[utf8]{inputenc}
\usepackage{amsmath}
\usepackage{amssymb}
\usepackage{amsfonts}
\usepackage{amsthm}
\usepackage{mathrsfs}
\usepackage{tikz-cd}
\usepackage{mathtools}
\usepackage{tensor}
\usepackage{geometry}
\usepackage{stmaryrd}
\usepackage{tikz}
\usepackage{pgfplots}\pgfplotsset{compat=1.13}
\usepackage{float}
\usepackage{framed, enumitem}
\usepackage{bm}
\usepackage{quiver}
\usepackage{hyperref}

\renewcommand{\sslash}{\mathbin{/\mkern-6mu/}}

\newcommand{\C}{\mathbb{C}}
\newcommand{\Z}{\mathbb{Z}}
\newcommand{\Q}{\mathbb{Q}}
\newcommand{\R}{\mathbb{R}}

\newcommand{\G}{\mathbb{G}}
\newcommand{\HH}{\mathcal{H}}
\newcommand{\J}{\mathcal{J}}
\newcommand{\OO}{\mathcal{O}}
\newcommand{\PP}{\mathbb{P}}

\newcommand{\F}{\mathcal{F}}
\newcommand{\E}{\mathcal{E}}
\newcommand{\LL}{\mathcal{L}}

\newcommand{\mm}{\mathfrak{m}}
\newcommand{\M}{\mathcal{M}}

\newcommand{\MM}{\mathfrak{M}}

\newcommand{\UU}{\mathcal{U}}
\newcommand{\Hom}{\mathcal{H}om}

\DeclareMathOperator{\Pic}{Pic}

\DeclareMathOperator{\spec}{Spec}
\DeclareMathOperator{\proj}{Proj}
\DeclareMathOperator{\sym}{Sym}
\DeclareMathOperator{\supp}{supp}

\newtheorem{theorem}{Theorem}[subsection]
\newtheorem{lemma}[theorem]{Lemma}
\newtheorem{prop}[theorem]{Proposition}
\newtheorem{cor}[theorem]{Corollary}
\newtheorem{defn}[theorem]{Definition}
\newtheorem{thmx}{Theorem}

\theoremstyle{definition}

\numberwithin{equation}{section}

\theoremstyle{remark}
\newtheorem*{remark}{Remark}

\title[GIT Constructions of Compactified Universal Jacobians over $\overline{\M}_{g,n}(X, \beta)$]{GIT Constructions of Compactified Universal Jacobians over Stacks of Stable Maps}
\author{George Cooper}
\address{Centro di Ricerca Ennio De Giorgi, Collegio Puteano, Scuola Normale Superiore, Piazza dei Cavalieri, 3, I-56100 PISA}
\email{george.cooper@sns.it}

\begin{document}
	
	
	\begin{abstract}
		We prove that any compactified universal Jacobian over the stack $\overline{\M}_{g,n}(X, \beta)$ defined using torsion-free sheaves which are Gieseker semistable with respect to a relatively ample invertible sheaf over the universal curve admits a projective good moduli space which can be constructed using GIT, and that the same is true for analogues parametrising semistable sheaves of higher rank. We also prove that for different choices of invertible sheaves, the corresponding good moduli spaces are related by a finite number of ``Thaddeus flips". As a special case of our methods, we provide a new GIT construction of the universal Picard variety of Caporaso and Pandharipande.
	\end{abstract}
	
	\maketitle
	
	\setcounter{tocdepth}{1}
	\tableofcontents
	
	\section{Introduction}
	
	This paper studies compactified universal Jacobians over the stack of stable maps $\overline{\M}_{g,n}(X, \beta)$, along with their analogues involving semistable torsion-free sheaves of rank $r > 1$, where the stability condition is defined using an ample invertible sheaf $\LL_{\UU}$ on the universal curve over $\overline{\M}_{g,n}(X, \beta)$. The main result is that these stacks admit projective good moduli spaces which are Geometric Invariant Theory (GIT) quotients, and that variation of GIT (VGIT) applies when the sheaf $\LL_{\UU}$ is varied. 
	
	A lot of recent attention has been given to the study of compactified universal Jacobians over the stack $\overline{\M}_{g,n}$ of $n$-pointed Deligne-Mumford stable curves of genus $g$ (\cite{abreupacinitrop} \cite{abreupagani} \cite{bfmv} \cite{caporasoneron} \cite{caporasochrist} \cite{casalainakassvivianilocal} \cite{casalainakassviviani} \cite{dudin} \cite{holmeskasspagani} \cite{kptheta} \cite{kasspagani} \cite{melopicard} \cite{melopicardmarked} \cite{melocuj} \cite{tropunivjacobian} \cite{meloviviani} \cite{paganitommasi}). These are compactifications of the stack of all line bundles of a given degree over smooth marked curves whose boundary objects consist of rank one torsion-free sheaves over singular marked stable curves.\footnote{Strictly speaking, the condition that the sheaves in the boundary are simple is also imposed, however we \emph{do not} make this requirement.} These stacks have been successfully used to construct extensions of the universal Abel map from from $\M_{g,n}$ to $\overline{\M}_{g,n}$, which in turn provides one way of defining the \emph{double ramification cycle} on $\overline{\M}_{g,n}$ (for more details on this connection see for instance the introduction of \cite{abreupacini} and the references therein).
	
	In order to compactify the stack of all line bundles over smooth curves (or more generally all slope semistable locally free sheaves), in practice a stability condition needs to be chosen. One way of doing this is to fix a relatively ample invertible sheaf $\LL_{\UU}$ on the universal curve over $\overline{\M}_{g,n}$ and then require that all torsion-free sheaves in the boundary are Gieseker semistable with respect to $\LL_{\UU}$. Melo in \cite{melocuj} proved that if all $\LL_{\UU}$-semistable sheaves are $\LL_{\UU}$-stable then the resulting compactified universal Jacobian admits a projective coarse moduli space. Melo's argument makes use of the criterion of Kollár \cite{kollar} to prove that the coarse moduli spaces are projective, and does not address the case where there are strictly $\LL_{\UU}$-semistable sheaves.
	
	If $n = 0$ and if $\LL_{\UU}$ is the dualising sheaf of the universal family, the resulting compactified universal Jacobian is also known to admit a projective good moduli space (even though there are strictly semistable sheaves), known as the \emph{universal Picard variety}; at least two GIT constructions of this space are known, the first due to Caporaso \cite{caporaso} and the second due to Pandharipande \cite{pand} (Pandharipande deals with the case where the sheaves are allowed to have any rank $r$, not just $r = 1$). The constructions of Caporaso and Pandharipande make heavy use of the GIT construction of the coarse moduli space $\overline{M}_g$ of Deligne-Mumford stable curves due to Gieseker \cite{gieseker}. In place of Gieseker's construction, we use Baldwin and Swinarski's GIT construction of the coarse moduli space $\overline{M}_{g,n}(X, \beta)$. We also use the GIT construction of Greb, Ross and Toma \cite{grebrosstoma} \cite{grebrosstomamaster} of moduli spaces of multi-Gieseker semistable sheaves on a projective scheme, which allows us to deal with multiple stability conditions at the same time. As a special case of our approach, we provide a new GIT construction of the Caporaso-Pandharipande moduli space. 

We also use the methods of \cite{grebrosstoma} and \cite{grebrosstomamaster} to establish the existence of rational linear wall-chamber decompositions on positive-dimensional spaces of stability conditions, with the property that as the stability condition varies within the interior of the chamber, the resulting stack remains unchanged, and that crossing a wall corresponds to modifying the good moduli space by a Thaddeus flip (also known as a VGIT flip). Unlike the wall-chamber decomposition described by Kass and Pagani in \cite{kasspagani} (in the case where the sheaves are all of rank one and where the base stack is $\overline{\M}_{g,n}$), the wall-chamber decompositions described in this paper rely on auxiliary choices of relatively ample invertible sheaves over the universal curve over the base stack $\overline{\M}_{g,n}(X , \beta)$, and as such the decompositions are not intrinsic to the stack $\overline{\M}_{g,n}(X , \beta)$.
	
	\subsection*{Summary of Results} Throughout we work over the field of complex numbers $\C$. Let $\overline{\M}_{g,n}(X, \beta)$ be the stack of genus $g$ stable maps with $n$ marked points to a projective variety $X$ of class $\beta$, and let $\M_{g,n}(X, \beta)$ be the open substack where the source curves are non-singular. Here and throughout this paper we assume that the discrete invariants are chosen so that the stack $\M_{g,n}(X, \beta)$ is non-empty. 
	
	Given a relatively ample $\Q$-invertible sheaf $\LL_{\UU}$ on the universal curve $\pi_{\UU} : \UU \overline{\M}_{g,n}(X, \beta) \to \overline{\M}_{g,n}(X, \beta)$, define $\overline{\J}_{g,n}^{d,r,ss}(X, \beta)(\LL_{\UU})$ to be the $\G_m$-rigidification of the stack parametrising flat and proper families of degree $d$, uniform rank $r$, torsion-free coherent sheaves over objects of $\overline{\M}_{g,n}(X, \beta)$ which are fibrewise Gieseker semistable with respect to $\LL_{\UU}$ (for more details see Section \ref{section: introducing the stacks}). Fix relatively ample $\Q$-invertible sheaves $\LL_1, \dots, \LL_k$, and set $\Sigma := (\Q^{\geq 0})^k \setminus \{0\}$. For each $\sigma = (\sigma_1, \dots, \sigma_k) \in \Sigma$,\footnote{Since $\lambda \sigma$ defines the same stability condition as $\sigma$ for each positive rational number $\lambda$, we are free to assume that $\sum \sigma_i = 1$ as necessary.} form the $\Q$-invertible sheaf $\LL_{\sigma} = \bigotimes_i \LL_i^{\sigma_i}$, and consider the resulting stacks over $\overline{\M}_{g,n}(X, \beta)$
	$$ \overline{\J}(\sigma) := \overline{\J}_{g,n}^{d,r,ss}(X, \beta)(\LL_{\sigma}). $$
	We say $\sigma \in \Sigma$ is \emph{positive} if each $\sigma_i > 0$, and \emph{degenerate} otherwise. Next, fix a finite subset $\mathfrak{S} \subset \Sigma$. We may now state our main result, which follows from Corollary \ref{cor: quotient stack description 2}, Theorem \ref{thm: good moduli spaces as GIT quotients} and Proposition \ref{prop: description of the closed points}.
	
	\begin{thmx} \label{thm: Theorem A}
		There exists a quasi-projective scheme $Z_r = Z_{r, \mathfrak{S}}$ and a reductive group $K$ acting on $Z_r$, such that for each $\sigma \in \mathfrak{S}$, there is a $K$-invariant open subset $Z_r^{\sigma-ss} \subset Z_r$, obtained as a GIT semistable locus for an appropriate linearisation determined by $\sigma$, which admits a good quotient $Z_r^{\sigma-ss} \sslash K$. Moreover:
		\begin{enumerate}
			\item if $\sigma$ is positive, then $\overline{\J}(\sigma)$ is isomorphic to the quotient stack $[Z_r^{\sigma-ss} / K]$;
			\item if $\sigma$ is either positive or degenerate, the stack $\overline{\J}(\sigma)$ admits a good moduli space $\overline{J}(\sigma) = \overline{J}_{g,n}^{d,r,ss}(X, \beta)(\LL_{\sigma})$, which is isomorphic to the good quotient $Z_r^{\sigma-ss} \sslash K$, and admits a natural morphism to the coarse moduli space $\overline{M}_{g,n}(X, \beta)$ of $\overline{\M}_{g,n}(X, \beta)$; and
			\item the fibre of $\overline{J}(\sigma)$ over $\zeta = [(C \to X; x_1, \dots, x_n)] \in M_{g,n}(X, \beta)$ is given by the quotient of the moduli space of degree $d$, rank $r$, slope-semistable vector bundles on $C$ by the action of the automorphism group of the stable map $\zeta$.
		\end{enumerate}
	\end{thmx}

\begin{remark}
	In the case where $\sigma$ is degenerate, it is still possible to use Theorem \ref{thm: Theorem A} to exhibit $\overline{\J}(\sigma)$ as a quotient stack of the form $[(Z_r')^{\sigma'-ss}/K']$ for an appropriate subgroup $K' \subset K$, by first omitting the invertible sheaves $\LL_i$ for which the corresponding entry $\sigma_i = 0$, letting $\sigma'$ be the vector obtained by omitting all zero entries from $\sigma$; see the remark following Corollary \ref{cor: quotient stack description 2}.
\end{remark}

As an application of the construction underlining Theorem \ref{thm: Theorem A} we prove (by adapting the arguments given in \cite{grebrosstomamaster}) the following result, concerning what happens at the level of the moduli spaces $\overline{J}(\sigma)$ as the stability condition $\sigma \in \Sigma$ is allowed to vary, after the $\Q$-invertible sheaves $\LL_1, \dots, \LL_k$ have been fixed. This result follows from Proposition \ref{prop: wall chamber structure} and Theorem \ref{thm: mumford thaddeus principle}.

\begin{thmx} \label{thm: Theorem B}
	The set $\Sigma' = \{\sigma \in \Sigma : \sum \sigma_i = 1\}$ is cut into chambers by a finite number of rational linear walls,\footnote{We allow for the possibility that all of $\Sigma'$ is a wall.} such that the moduli stack $\overline{\J}(\sigma)$ is unchanged as $\sigma$ varies in the interior of a chamber. For any $\sigma \in \Sigma'$, all $\LL_{\sigma}$-semistable sheaves are $\LL_{\sigma}$-stable, and hence the corresponding stack $\overline{\J}(\sigma)$ is Deligne-Mumford, if and only if $\sigma$ does not lie in any wall. Given $\sigma_1$, $\sigma_2 \in \Sigma'$, the moduli spaces $\overline{J}(\sigma_i)$ ($i = 1,2$) are related by a finite number of Thaddeus flips (cf. Definition \ref{defn: thaddeus flip}) through moduli spaces of the form $\overline{J}(\sigma')$, $\sigma' \in \Sigma'$.
\end{thmx}

\subsection*{Comparison with Other Work} In the case where $r = 1$ the stacks $\overline{\J}_{g,n}^{d,1,ss}(X, \beta)(\LL_{\UU})$ are examples of compactified universal Jacobians. The paper \cite{melocuj} treats these stacks in the case where $X$ is a point and $\overline{\M}_{g,n}(X, \beta)$ is replaced with the stack of $n$-pointed genus $g$ curves, or more generally any open substack of the stack of $n$-pointed genus $g$ prestable curves. In particular, Melo considers stability conditions defined with respect to a vector bundle on the universal curve $\pi_{\UU} : \UU \overline{\M}_{g,n} \to \overline{\M}_{g,n}$, generalising the stability conditions introduced by Esteves in \cite{estevesreljacobian}, and for explicit polarisations determines functoriality properties of the corresponding compactified universal Jacobians, such as compatibility with forgetful and clutching morphisms.

Meanwhile, Kass and Pagani \cite{kasspagani} study stability conditions defined in terms of suitable functions $\phi$ defined on dual graphs of stable curves in $\overline{\M}_{g,n}$, generalising the approach of Oda and Seshadri \cite{odaseshadri}, and describe an explicit wall-chamber decomposition on the space of stability conditions, which in their notation is $V_{g,n}^d$, where $d$ is the degree of the rank $1$ torsion-free sheaves under consideration. It should be noted that Kass-Pagani and Melo study the same family of stability conditions (see \cite[Remark 4.6]{kasspagani}), and that each of the Kass-Pagani stability conditions can be realised as a (twisted) Gieseker stability condition with respect to a relatively ample invertible sheaf on the universal curve over $\overline{\M}_{g,n}$ (see \emph{loc. cit.} Corollary 4.3), and so, by the final remark of Section \ref{section: substacks defined by polarisations}, are covered by our results. Kass and Pagani also successfully characterise for any $\phi \in V_{g,n}^d$ the locus of indeterminacy of the extension $\overline{\M}_{g,n} \dashrightarrow \J_{g,n}^d(\phi)$ (where $\J_{g,n}^d(\phi)$ denotes Kass and Pagani's compactified universal Jacobian) of the Abel-Jacobi map
$$ \alpha_{k,\underline{d}} : (C; x_1, \dots, x_n) \mapsto \omega_C^{-k}(d_1x_1 + \cdots + d_nx_n), \quad (C; x_1, \dots, x_n) \in \M_{g,n}, $$
and in particular describe all stability conditions $\phi$ for which $\alpha_{k,\underline{d}}$ extends to $\overline{\M}_{g,n}$.

The main difference between the decomposition of $V_{g,n}^d$ and the decompositions described by Theorem \ref{thm: Theorem B} is that the latter decompositions depend on the choice of the $\Q$-invertible sheaves $\LL_1, \dots, \LL_k$, whereas the decomposition of $V_{g,n}^d$ only depends on the discrete invariants $g, n$ and $d$. It would be interesting to determine how the wall-chamber decompositions relate to each other (given explicit $\Q$-invertible sheaves $\LL_1, \dots, \LL_k$), as well as whether the decomposition of \cite{kasspagani} extend to higher rank sheaves and to stable maps.

It is possible to construct the good moduli space of $\overline{\J}_{g,n,d,r}^{ss}(X, \beta)(\LL)$ via GIT in a manner which closely follows the construction of Pandharipande \cite{pand}, by applying relative GIT to an appropriate Quot scheme over the parameter space considered by Baldwin--Swinarski \cite{baldwinswinarski} (in place of the parameter space considered by Gieseker \cite{gieseker}), making use of the results of Simpson \cite{simpson} concerning GIT constructions of relative moduli spaces of Gieseker semistable sheaves. We remark that the difficult step in Pandharipande's construction is relating GIT and moduli (semi)stability for the fibrewise action on the relative Quot scheme; once this has been established, the rest of Pandharipande's construction is essentially formal relative GIT. As observed by Pandharipande himself (cf. \cite[Page 433]{pand}), the work of Simpson can instead be used to fully solve the fibrewise GIT problem. 
	
The disadvantage of this approach, as compared with that used to prove Theorem \ref{thm: Theorem A}, is that it is not possible to implement VGIT, since different choices of linearisations $\mathcal{L}$ necessitate working with different relative Quot schemes.

For higher ranks, there are alternative compactifications of moduli spaces of slope-semistable vector bundles over the moduli stack of smooth curves, utilising vector bundles on semistable curves instead of semistable torsion-free sheaves on stable curves (see for instance \cite{fringuelli}, \cite{schmitthilbert} and \cite{teixidor}). We do not consider these compactifications in this paper.

\subsection*{Outline of the Paper}

After covering the preliminaries in Section \ref{section: preliminaries}, we introduce in Section \ref{section: introducing the stacks} the stacks $\overline{\J}_{g,n}^{d,r,ss}(X, \beta)(\LL_{\UU})$, and show that in the case where all $\LL_{\UU}$-semistable sheaves are $\LL_{\UU}$-stable, these stacks are Deligne-Mumford and admit perfect relative obstruction theories. We also summarise some basic properties of the stacks $\overline{\J}_{g,n}^{d,r,ss}(X, \beta)(\LL_{\UU})$. In Section \ref{section: BS GRT constructions} we sketch the GIT constructions of Baldwin-Swinarski \cite{baldwinswinarski} and Greb-Ross-Toma \cite{grebrosstoma} \cite{grebrosstomamaster}, indicating how the latter construction extends into a relative setting. At the end of Section \ref{section: BS GRT constructions} we prove a result concerning the existence of further GIT quotients of the moduli space of Greb, Ross and Toma, which allows us to combine together the aforementioned GIT constructions. This is carried out in Section \ref{section: carrying out the construction}, where in the same section we use this to prove Theorem \ref{thm: Theorem A}. We also give in this section a full description of the closed points of the resulting good moduli spaces. Finally in Section \ref{section: wall chamber decompositions} we explain how the master spaces $Z_r$ appearing in the statement of \ref{thm: Theorem A} can be used to prove Theorem \ref{thm: Theorem B}.

\subsection*{Remarks on the Base Field} We have chosen to work with the field $\C$ for simplicity, though our results are valid over any algebraically closed field $k$ of characteristic zero. Indeed, the papers \cite{baldwinswinarski}, \cite{grebrosstoma} and \cite{grebrosstomamaster} all assume that the base field is algebraically closed and of characteristic zero. It is not yet known whether the results of \cite{grebrosstoma} and \cite{grebrosstomamaster} hold in arbitrary characteristic, and the GIT construction of $\overline{M}_{g,n}(X, \beta)$ fails if the characteristic of the base field is non-zero and not sufficiently large (see \cite{baldwinpositive}).

\subsection*{Acknowledgements}

The contents of this paper constitutes a chapter of the author's DPhil thesis at the University of Oxford. The author wishes to thank his DPhil supervisors Frances Kirwan, Jason Lotay and Alexander Ritter, for introducing the author to the problem and for their numerous helpful suggestions and comments, without which this work would not be possible. The author also wishes to thank his thesis examiners Ruadhai Dervan and Dominic Joyce for their corrections and comments on this work; thanks is also due to Andrés Ibáñez Núñez, Henry Liu, Margarida Melo, Nicola Pagani, Dhruv Ranganathan, Julian Ross and Montserrat Teixidor i Bigas for their interest in this work alongside helpful conversations and suggestions.

The author was supported by an EPSRC Doctoral Scholarship (EP/V520202/1).

\subsection*{Competing Interests}

The author declare none.

\subsection*{Conventions and Notation} Throughout this paper we work over the field $\C$ of complex numbers. A \emph{point} of a scheme $S$ is a closed point $s : \spec \C \to S$. All varieties are assumed to be separated, irreducible and of finite type. All group schemes are assumed to be smooth, separated and of finite type.

As in \cite{alvarezconsulking} and \cite{grebrosstoma}, we use the notation
$$ \begin{cases} \text{for } n \gg 0 & \text{to mean } \exists n_0 : \forall n \geq n_0, \\ \text{for } m \gg n \gg 0 & \text{to mean } \exists n_0 : \forall n \geq n_0, \exists m_0 : \forall m \geq m_0, \end{cases} $$
and so on. If $p(t), q(t) \in \R[t]$ are polynomials, we write $p < q$ if for all $m \gg 0$ we have $p(m) < q(m)$, and similarly for $p \leq q$.

We often use the symbol $\mathcal{U}$ to denote universal objects over moduli stacks; for example, $\pi_{\UU} : \UU \overline{\M}_{g,n} \to \overline{\M}_{g,n}$ is the universal curve over the stack $\overline{\M}_{g,n}$. 

A \emph{$\Q$-line bundle} is a formal tensor power $L = N^{\alpha}$, where $N$ is a line bundle and where $\alpha$ is a rational number; $L$ is said to be \emph{ample} if $N$ is ample and if $\alpha > 0$. The \emph{degree} of $L$ is defined by $\deg L = \alpha \deg M$. We make no distinction between line bundles and invertible sheaves.

Suppose $X$ is a projective scheme and $\underline{L} = (L_1, \dots, L_k)$ is a tuple of ample invertible sheaves on $X$. Given a coherent sheaf $E$ over $X$, the \emph{topological type} $\tau = \tau(E)$ of $E$ (with respect to $\underline{L}$) is the tuple $(P_1(t), \dots, P_k(t))$, where $P_j(t)$ is the Hilbert polynomial $\chi(C, E \otimes L_j^t)$ of $E$ with respect to $L_j$. In the case where $X$ is projective over a locally Noetherian base scheme $S$ and we are given relatively ample invertible sheaves $\underline{\LL} = (\LL_1, \dots, \LL_k)$ on $X$, we extend the notion of topological type to $S$-flat sheaves on $X$ in the obvious way.\footnote{This is slightly different to the notion of topological type in \cite{grebrosstoma} and \cite{grebrosstomamaster}; in these papers the topological type of a coherent sheaf is defined to be the homological Todd class $\tau_X(E) \in B(X)_{\Q}$ of $E$ (though see \cite[Remark 1.5]{grebrosstoma}).}

A \emph{curve} is a connected, reduced, projective scheme of pure dimension $1$. The \emph{genus} $g$ of $C$ is the arithmetic genus of $C$. The dualising sheaf of $C$ is denoted $\omega_C$. $C$ is \emph{nodal} if every $p \in \mathrm{sing}(C)$ is a node, that is $\widehat{\OO}_{C,p} \cong \C[[x,y]]/(xy)$. If $(C; x_1, \dots, x_n)$ is a marked curve, we say $C$ is \emph{prestable} if $C$ is nodal and if the markings are distinct and non-singular. A \emph{subcurve} $D$ of a curve $C$ is a closed subscheme of $C$ which is reduced, of pure dimension $1$ (but not necessarily connected). $D$ is a \emph{proper subcurve} of $C$ if $D$ is non-empty and not the whole of $C$. The \emph{complementary subcurve} to $D$ is $D^c = \overline{C \setminus D}$. We set $k_D = |D \cap D^c|$.

	A coherent sheaf $F$ on a curve $C$ is:
\begin{itemize}
	\item of \emph{rank $r$} if $F$ has rank $r$ at the generic point of every component of $C$;
	\item \emph{torsion-free} if $\supp F = C$, and if for every non-zero subsheaf $G \subset F$ one has $\dim \supp G = 1$;
	\item \emph{simple} if $\mathrm{End}(F) = \C$.
\end{itemize}
If $F$ is a rank $r$ torsion-free sheaf on $C$ and $D$ is a subcurve of $C$, we let $F_D$ denote the restriction of $F$ to $D$ modulo torsion and set $\deg_D F = \chi(F_D) - r\chi(\OO_D)$. We define $\deg F = \deg_C F$.
If $C_1, \dots, C_{\rho}$ are the irreducible components of $C$, the \emph{multidegree} of $F$ is the tuple $(\deg_{C_1} F, \dots, \deg_{C_{\rho}} F)$. We have the relation $\deg F = \sum_{j=1}^{\rho} \deg_{C_j} F - \delta(F)$, where $\delta(F)$ is the number of nodes where $F$ is not locally free.

A \emph{family of curves} is a proper surjective flat morphism $f : \mathcal{C} \to S$ of finite presentation whose geometric fibres are curves. A \emph{family of coherent sheaves} on a family of curves $f : \mathcal{C} \to S$ is a coherent sheaf on $\mathcal{C}$ which is $S$-flat and of finite presentation.

Following \cite{behrendmanin}, a \emph{curve class} on a projective variety $X$ is an element of the semigroup
$$ H_2(X, \Z)_+ = \{ \beta \in \mathrm{Hom}_{\Z}(\Pic X, \Z) : \beta(L) \geq 0 \text{ for all ample } L \}. $$
If $f : C \to X$ is a morphism from a prestable curve $C$, the locally constant function $L \mapsto \deg f^{\ast} L$ on $\Pic X$ is denoted as $f_{\ast}[C] \in H_2(X,\Z)_+$.

Whenever we work with stacks of stable maps $\overline{\M}_{g,n}(X, \beta)$, we implicitly assume that the discrete invariants $g$, $n$ and $\beta$ are chosen in such a way that the open substack $\M_{g,n}(X, \beta)$ is always non-empty. 

We sometimes abbreviate ``category fibred in groupoids" as CFG.

\section{Preliminaries} \label{section: preliminaries}

\subsection{Good Moduli Spaces of Quotient Stacks} Let $\mathcal{X}$ be an algebraic stack over $\C$. We review Alper's notion of a \emph{good moduli space} of $\mathcal{X}$, and how it relates to GIT over general base schemes.

\begin{defn}[\cite{alpergood}]
	A morphism $\phi : \mathcal{X} \to X$ is said to be a \emph{good moduli space} if $X$ is an algebraic space and if
	\begin{enumerate}
		\item $\phi_{\ast} : \mathbf{QCoh}(\mathcal{X}) \to \mathbf{QCoh}(X)$ is exact; and
		\item the natural morphism $\OO_X \to \phi_{\ast} \OO_{\mathcal{X}}$ is an isomorphism.
	\end{enumerate}
\end{defn}

A good moduli space $\phi : \mathcal{X} \to X$ enjoys the following properties:
\begin{enumerate}
	\item $\phi$ is surjective and universally closed, and $X$ has the quotient topology.
	\item Two points $x_1, x_2 \in \mathcal{X}(\C)$ map to the same point of $X$ if and only if their closures meet.
	\item If $X' \to X$ is a morphism of algebraic spaces then the base change $\mathcal{X} \times_X X' \to X'$ is a good moduli space.
	\item If $\mathcal{X}$ is locally Noetherian then $\phi$ is universal for morphisms to algebraic spaces; if moreover $\mathcal{X}$ is Deligne-Mumford then $\phi : \mathcal{X} \to X$ is a coarse moduli space.
	\item If $\mathcal{X}$ is of finite type over a scheme $S$ then so is $Y$.
\end{enumerate}

\begin{prop}[\cite{alpergood}] \label{prop: good moduli spaces and good quotients}
	If $\mathcal{X} = [Z/G]$, where $G$ is a reductive group acting on a quasi-projective scheme $Z$, then an algebraic space $X$ is the good moduli space of $[Z/G]$ if and only if $X$ is the good quotient of $Z$ by $G$.
\end{prop}

In order to state the link between the theory of good moduli spaces and GIT, the definition of GIT semistability given in \cite{mfk} needs to be generalised first.

\begin{defn}[\cite{alpergood}, Definition 11.1] \label{defn: relative semistable locus}
	Let $S$ be a quasi-compact scheme, let $G$ be a reductive group scheme over $S$, let $p : X \to S$ be a quasi-projective morphism and let $L$ be a $G$-linearisation on $X$. A point $x \in X$ is said to be \emph{(relatively) semistable} (with respect to $L$) if there exists an open neighbourhood $U \subset S$ of $p(x)$, a positive integer $n$ and an invariant section $t \in H^0(p^{-1}(U), L^n)^G$ for which $t(x) \neq 0$ and for which the locus $\{y \in p^{-1}(U) : t(y) \neq 0\}$ is affine. The locus of points which are semistable with respect to $L$ is denoted $X^{ss}(L/S)$.\footnote{In the case where the base $S = \spec \C$, we use the usual notation $X^{ss}(L) = X^{ss}(L/\spec \C)$.}
\end{defn}

\begin{prop}[\cite{alpergood}, Theorem 13.6] \label{prop: alper and git}
	Let $S$ be a quasi-compact scheme, let $G$ be a reductive group scheme over $S$, let $p : X \to S$ be a quasi-projective morphism and let $L$ be a $G$-linearisation on $X$, with corresponding relative semistable locus $X^{ss}(L/S)$. Then there exists a good moduli space $\phi : [X^{ss}(L/S)/G] \to Y$, with $Y$ an open subscheme of $\mathbf{Proj}_S \bigoplus_{k=0}^{\infty} (p_{\ast} L^k)^G$. Moreover, there exists an $S$-ample invertible sheaf\footnote{Here, $M$ can be taken to be the restriction of $\OO(N)$ for some $N > 0$, where $\OO(1)$ denotes the twisting sheaf.} $M$ on $Y$ such that $q^{\ast}(M) \cong L^N|_{X^{ss}(L/S)}$ for some $N > 0$. If $p$ is projective and if $L$ is relatively ample, then $Y$ is projective over $S$.
\end{prop}

The following result will be used several times in Section \ref{section: BS GRT constructions}.

\begin{lemma} \label{lem: technical git over base result}
	Let $S$ be a quasi-compact scheme, let $p : X \to S$ be an affine morphism of finite type, let $G$ be a reductive linear algebraic group (over $\C$), and let $\chi$ be a character of $G$. Suppose $G$ acts on $X$ (and acts trivially on $S$) in such a way that the morphism $p$ is invariant. Let $L = \OO_X(\chi)$ be the linearisation given by twisting the structure sheaf $\OO_X$ by the character $\chi$, and let $X^{ss}(L/S)$ denote the resulting relatively semistable locus. Assume that for each geometric point $s \in S$, the (non-relative) semistable locus $(X_s)^{ss}(L_s)$ for the induced action $G \circlearrowright X_s$ is non-empty. 
	
	Then the good moduli space of $[X^{ss}(L/S)/G]$ is given by
	$$ Y = \mathbf{Proj}_S \bigoplus_{k=0}^{\infty} (p_{\ast} L^k)^G = \mathbf{Proj}_S \bigoplus_{k=0}^{\infty} (p_{\ast} \OO_X(\chi^k))^G, $$
	and for each geometric point $s \in S$ there is an equality of schemes
	\begin{equation} \label{eqn: fibrewise semistable loci}
		(X^{ss}(L/S))_s = (X_s)^{ss}(L_s).
	\end{equation}
\end{lemma}

\begin{proof}
	There is a map $X \dashrightarrow \mathbf{Proj}_S \bigoplus_{k=0}^{\infty} (p_{\ast} L^k)^G$ arising from the inclusions of $G$-invariants $(p_{\ast} L^k)^G \subset p_{\ast} L^k$; the domain of definition of this map is the open subscheme $X^{ss}(L/S)$, and the image coincides with the good moduli space $Y$.
	
	Take a geometric point $s \in S$. We have an inclusion $(X^{ss}(L/S))_s \subset (X_s)^{ss}(L_s)$. To show that the reverse inclusion holds, without loss of generality we may assume that $S = \spec A$ is affine, so that $X = \spec B$ is also affine. Let $\mm \subset A$ be the maximal ideal which corresponds to $s \in S$, so $X_s = \spec B/\mm B$. If $x \in (X_s)^{ss}(L_s)$ is a point, there exists a positive integer $n$ and a $\chi^n$-semi-invariant $\overline{f} \in B / \mm B$ which does not vanish at $x$. The ring of $\chi$-semi-invariants of $B / \mm B$ can be identified with a ring of $G$-invariants of the graded ring $(B/\mm B)[t]$ (cf. \cite[Page 193]{mukaimoduli}). As $G$ is reductive then taking $G$-invariants of graded rings is exact; consequently there exists a $\chi^n$-semi-invariant $f \in B$ whose image in $B / \mm B$ is $\overline{f}$. In particular, $f$ does not vanish at $x$, and so $x \in (X^{ss}(L/S))_s$; this gives the equality \eqref{eqn: fibrewise semistable loci}.
	
	It remains to show $Y = \mathbf{Proj}_S \bigoplus_{k=0}^{\infty} (p_{\ast} L^k)^G$. However, the restriction of the morphism $X^{ss}(L/S) \to \mathbf{Proj}_S \bigoplus_{k=0}^{\infty} (p_{\ast} L^k)^G$ over a geometric point $s \in S$ coincides with the surjective good quotient
	
	$$ (X_s)^{ss}(L_s) \to \left( \mathbf{Proj}_S \bigoplus_{k=0}^{\infty} (p_{\ast} L^k)^G \right) \times_S \spec \C = \proj \bigoplus_{k=0}^{\infty} H^0(X_s, \OO_{X_s}(\chi^k))^G, $$
	and so $X^{ss}(L/S) \to \mathbf{Proj}_S \bigoplus_{k=0}^{\infty} (p_{\ast} L^k)^G$ must be surjective.
\end{proof}

\subsection{Thaddeus Flips}

Here we recall the notion of a Thaddeus flip (in the relative setting), following \cite{grebrosstoma} and \cite{grebrosstomamaster}.

\begin{defn} \label{defn: thaddeus flip}
	Let $Y_+$, $Y_-$ and $Y_0$ be schemes. We say that $Y_+$ and $Y_-$ are \emph{related by a Thaddeus flip through $Y_0$} if there exists a Noetherian scheme $X$ over a quasi-compact scheme $S$, acted on by a reductive linear algebraic group $G$ over $S$, with $G$-linearisations $L_+, L_-, L_0$ on $X$ such that:
	\begin{enumerate}
		\item there are equalities $Y_{\pm} = X^{ss}(L_{\pm}/S) \sslash G$ and $Y_0 = X^{ss}(L_0/S) \sslash G$; and
		\item there exists a diagram
		\[\begin{tikzcd}[ampersand replacement=\&,cramped]
			{Y_+} \&\&\&\& {Y_-} \\
			\&\& {Y_0}
			\arrow["{\psi_+}"', from=1-1, to=2-3]
			\arrow["{\psi_-}", from=1-5, to=2-3]
			\arrow[dotted, from=1-1, to=1-5]
			\arrow[dotted, from=1-5, to=1-1]
		\end{tikzcd}\]
		where $\psi_{\pm}$ is the morphism arising from an inclusion of relative semistable loci $X^{ss}(L_{\pm}/S) \subset X^{ss}(L_0/S)$.
	\end{enumerate}
\end{defn}

The geometry of a Thaddeus flip is governed by the results of \cite{thaddeus} and \cite{dolgachevhu}\footnote{Thaddeus works in the absolute setting $S = \spec \C$, however the above stated results concern properties which are affine-local over $Y_0 = X^{ss}(L_0/S) \sslash G$, and so carry over to the relative setting.}. Note that unlike in other occurrences of notions of flips we do not require the map $Y_+ \dashrightarrow Y_-$ to be birational.

\subsection{A Miscellaneous Relative GIT Result}

The following result will be used to help establish Proposition \ref{prop: further quotients}.

Let $H$ be a reductive group, let $X$, $Y$ be quasi-projective schemes, and let $q : X \to Y$ be a projective morphism. Assume $X$ and $Y$ admit $H$-actions with respect to which the morphism $q$ is equivariant. Let $\mathcal{N}$ be a $H$-linearised relatively ample line bundle on $X$. Assume there exists an equivariant open immersion $\iota : Y \hookrightarrow \overline{Y}$ into a projective scheme $\overline{Y}$ acted on by $H$ with ample linearisation $L$. Assume further that
$$ \overline{Y}^{ss}(L) = \overline{Y}^s(L) = Y. $$ 

\begin{prop} \label{prop: relative GIT result}
	Let $q : X \to Y$ be as above. Then there are projective geometric quotients $X \sslash H$ and $Y \sslash H$ of $X$ and $Y$ respectively, and the morphism $q$ induces a projective morphism $\hat{q} : X \sslash G \to Y \sslash H$. If $y \in Y$ is a closed point, the fibre of $\hat{q}$ over the orbit $G \cdot y$ is isomorphic to $X_y / H_y$, the geometric quotient of $X_y$ by the stabiliser $H_y$ of the point $y \in Y$.
\end{prop}

\begin{proof}
	In the situation of Proposition \ref{prop: relative GIT result}, there are representations $H \to GL(V_i)$, $i = 1, 2$, a commutative diagram
	\[\begin{tikzcd}
		X &&&& {\PP(V_1) \times \PP(V_2)} \\
		\\
		Y && {\overline{Y}} && {\PP(V_2)}
		\arrow["\kappa", hook, from=1-1, to=1-5]
		\arrow["q"', from=1-1, to=3-1]
		\arrow["\iota", hook, from=3-1, to=3-3]
		\arrow["\eta", hook, from=3-3, to=3-5]
		\arrow["{\mathrm{pr}_2}", from=1-5, to=3-5]
	\end{tikzcd}\]
	whose rows are given by $H$-equivariant locally closed immersions, and positive integers $a$ and $b$ with
	$$ N^{a} \cong \kappa^{\ast}(\OO_{\PP(V_1) \times \PP(V_2)}(1,0)) \text{   and   } L^{b} \cong \eta^{\ast}(\OO_{\PP(V_2)}(1)). $$
	Let $\overline{X}$ denote the closure of the image of $X$ in $\PP(V_1) \times \PP(V_2)$. The morphism $\mathrm{pr}_2 : \overline{X} \to \PP(V_2)$ factors through $\overline{Y}$, giving a projective morphism $\overline{q} : \overline{X} \to \overline{Y}$ making the following diagram commute:
	\[\begin{tikzcd}
		X && {\overline{X}} && {\PP(V_1) \times \PP(V_2)} \\
		\\
		Y && {\overline{Y}} && {\PP(V_2)}
		\arrow["q"', from=1-1, to=3-1]
		\arrow["\iota", hook, from=3-1, to=3-3]
		\arrow["\eta", hook, from=3-3, to=3-5]
		\arrow["{\mathrm{pr}_2}", from=1-5, to=3-5]
		\arrow["\kappa", hook, from=1-1, to=1-3]
		\arrow["\subset", hook, from=1-3, to=1-5]
		\arrow["{\overline{q}}"', from=1-3, to=3-3]
	\end{tikzcd}\]
	
	By \cite[Theorem 3.11]{hurelative}\footnote{As indicated by Schmitt in \cite{schmittrgitremark}, Hu's result is erroneously stated and only applies in certain circumstances; the case where both the domain and target are projective is one such circumstance.} for all $d \gg 0$ there are equalities of semistable loci (over $\spec \C$)
	\begin{align*}
		\overline{X}^{ss}(\OO_{\PP(V_1) \times \PP(V_2)}(1,d) |_{\overline{X}}) = \overline{X}^{s}(\OO_{\PP(V_1) \times \PP(V_2)}(1,d) |_{\overline{X}}) &= \overline{q}^{-1}(\overline{Y}^{ss}( \eta^{\ast} \OO_{\PP(V_2)}(1) )) \\ &= \overline{q}^{-1}(\overline{Y}^{ss}(L^{b}))  \\ &= \overline{q}^{-1}(Y) = X.
	\end{align*}
	
	Applying Theorem 3.13 of \emph{loc. cit.} to $\bar{q}$ yields the result.
\end{proof}

\subsection{Multi-Gieseker Stability and Multi-Regular Sheaves} We review the notion of \emph{multi-Gieseker stability} for coherent sheaves on projective schemes, as introduced in \cite{grebrosstoma}.

Let $X$ be a projective scheme. Fix a \emph{stability parameter} $\sigma = (\underline{L}, \sigma_1, \dots, \sigma_k)$ on $X$; here $\underline{L}$ is a tuple of very ample line bundles $L_1, \dots, L_k$, and the $\sigma_i$ are non-negative rational numbers, not all zero. Given a coherent sheaf $E$ on $X$ of dimension $d$, the \emph{multi-Hilbert polynomial} of $E$ with respect to $\sigma$ is the polynomial
$$ P_E^{\sigma}(t) = \sum_{j=1}^k \sigma_j \chi(X, E \otimes L_j^t) = \sum_{i=0}^d \alpha_i^{\sigma}(E) \frac{t^i}{i!}, $$
and the \emph{reduced multi-Hilbert polynomial} of $E$ is
$$ p_E^{\sigma}(t) = \frac{P_E^{\sigma}(t)}{\alpha_d^{\sigma}(E)}. $$

\begin{defn} \label{defn: multigieseker stability}
	A coherent sheaf $E$ on $X$ is said to be \emph{multi-Gieseker (semi)stable}\footnote{Where no confusion is likely to arise, we also refer to this notion as \emph{$\sigma$-(semi)stability.}} with respect to $\sigma$ if $E$ is pure, and if for all non-zero proper subsheaves $F \subset E$,
	$$ p_F^{\sigma} < (\leq) \ p_E^{\sigma}. $$
\end{defn}

\begin{remark}
	Definition \ref{defn: multigieseker stability} and similar statements should be understood as two separate statements, with one statement defining multi-Gieseker semistability (using $\leq$) and the other statement defining multi-Gieseker stability (using $<$).
\end{remark}

In the case where only one $\sigma_i$ is non-zero, we recover the usual notion of Gieseker (semi)stability with respect to the polarisation $L_i$. The notions of Harder-Narasimhan filtrations, Jordan-H\"older filtrations and S-equivalence carry over to the multi-Gieseker setting, and multi-Gieseker (semi)stability is an open property in flat families (cf. \cite[Section 2]{grebrosstoma}). If $E$ is $\sigma$-semistable, the associated graded sheaf is denoted $\mathrm{gr}_{\sigma}(E)$.

\begin{defn}
	The coherent sheaf $E$ is said to be \emph{$(m, \underline{L})$-regular} if $E$ is $m$-regular with respect to each $L_j$, that is for each $i > 0$ we have $H^i(E \otimes L_j^{m-i}) = 0$. More generally, if $\pi : X \to S$ is a projective morphism of schemes, if $\underline{\LL} = (\LL_1, \dots, \LL_k)$ is a tuple of $\pi$-very ample sheaves and if $\E$ is an $S$-flat coherent sheaf on $X$, we say that $\E$ is \emph{$(m, \underline{\LL})$-regular} if for all $j = 1, \dots, k$ and for all $i > 0$, one has $R^i \pi_{\ast}(\E \otimes \LL_j^{m-i}) = 0$.
\end{defn}

\begin{defn}
	Let $\tau$ be a topological type of sheaves on $X$ defined with respect to very ample line bundles $L_1, \dots, L_k$. The stability parameter $\sigma = (\underline{L}; \sigma_1, \dots, \sigma_k)$ is said to be
	\begin{itemize}
		\item \emph{positive} if each $\sigma_i > 0$;
		\item \emph{degenerate} if there exists some $\sigma_i = 0$; and
		\item \emph{bounded (with respect to $\tau$)} if the collection of all $\sigma$-semistable sheaves on $X$ of topological type $\tau$ forms a bounded family.
	\end{itemize}
\end{defn}

\begin{remark}
	If $\E$ is an $S$-flat, $(m, \LL)$-regular sheaf over a projective $S$-scheme $\pi : X \to S$ and if $S$ is locally Noetherian, then for all $m' \geq m$ and for all $j = 1, \dots, k$ the sheaves $\pi_{\ast}(\E \otimes \LL_j^{m'})$ are locally free; this follows from the cohomology and base change theorem of \cite{ega3_2} as well as basic properties of Castelnuovo-Mumford regularity.
\end{remark}

\subsection{Sheaves on Nodal Curves} \label{section: sheaves on nodal curves} Let $C$ be a nodal curve with components $C_1, \dots, C_{\rho}$. Fix very ample line bundles $L_1, \dots, L_k$ on $C$ and a stability condition $\sigma = (\underline{L}, \sigma_1, \dots, \sigma_k)$. If $E$ is a torsion-free sheaf whose rank at the generic point of $X_i$ is $r_i(E)$, by \cite[Septième Partie, Corollaire 8]{seshadri}, we have
\begin{equation} \label{eqn: MGS nodal curve}
	P_E^{\sigma}(t) = \sum_{j=1}^k \sigma_j \chi(X, E \otimes L_j^t) = \chi(E) \sum_{i=1}^k \sigma_i + t \sum_{i=1}^k \sigma_i \left( \sum_{j=1}^{\rho} r_j(E) \deg_{C_j} L_i \right).
\end{equation}
Consequently $E$ is $\sigma$-(semi)stable if and only if for all non-zero proper subsheaves $F \subset E$ one has $\mu^{\sigma}(F) < (\leq) \ \mu^{\sigma}(E)$, where
\begin{equation} \label{eqn: multi slope stability nodal curve}
	\mu^{\sigma}(E) = \frac{\chi(E)}{\sum_{i,j} \sigma_i r_j(E) \deg_{C_j} L_i}.
\end{equation}
An immediate consequence of Inequality \ref{eqn: multi slope stability nodal curve} is that the stability condition is unchanged if each $L_i$ is replaced by $L_i^{p}$, where $p$ is a positive integer. This allows us to extend the notion of $\sigma$-stability to the case where the $L_i$ are ample $\Q$-line bundles. Moreover, in the nodal curve case, $\sigma$-(semi)stability coincides with Gieseker (semi)stability with respect to the $\Q$-ample line bundle $\bigotimes_i L_i^{\sigma_i}$.

\begin{remark}
	Equation \eqref{eqn: MGS nodal curve} implies that fixing the genus $g$ of $C$ and the rank $r$ and the degree $d$ of $E$ fixes the topological type $\tau$ of $E$ with respect to any finite collection of ample line bundles on $C$. Conversely, if the genus of $C$ is fixed beforehand and if it is known that $E$ is a torsion-free sheaf of uniform rank, then the degree and rank of $E$ can be recovered from the topological type.
\end{remark}

\begin{remark}
	In the case $r = 1$ the condition of $\sigma$-stability may be rephrased in terms of the subcurves $D$ of $C$; a rank $1$ torsion-free sheaf $E$ on $C$ is $\sigma$-(semi)stable if and only if for all connected proper subcurves $D \subset C$,
	$$ \deg_D E > (\geq) \ \frac{\sum_i \sigma_i \deg_D L_i}{\sum_i \sigma_i \deg L_i} \left( \deg E - \frac{\deg \omega_C}{2} \right) + \frac{\deg_D \omega_C}{2} - \frac{k_D}{2}. $$
	This follows from essentially the same argument given in \cite[Section 1.4]{alexeev}.
\end{remark}

\subsection{Stable Maps} We briefly recall the definition of a stable map to a fixed projective variety $X$. Let $\beta \in H_2(X,\Z)_+$ be a fixed curve class, let $g, n$ be non-negative integers and let $\overline{\M}_{g,n}(X, \beta)$ be the Deligne-Mumford stack parametrising flat and proper families of $n$-marked, genus $g$ stable maps to $X$ of class $\beta$. The $\C$-points of $\overline{\M}_{g,n}(X, \beta)$ are given by objects $(C; x_1, \dots, x_n; f)$, where $(C; x_1, \dots, x_n)$ is an $n$-marked prestable curve of genus $g$ and where $f : C \to X$ is a morphism to $X$ with $f_{\ast}[C] = \beta$, satisfying any of the following (equivalent) conditions:
\begin{itemize}
	\item $(C; x_1, \dots, x_n; f)$ has no infinitesimal automorphisms.
	\item If $X \subset \PP^b$ is a fixed embedding of $X$ inside a projective space, then the line bundle $\omega_C(x_1 + \cdots + x_n) \otimes f^{\ast} (\OO_{\PP^b}(3)|_X)$ is ample.
	\item Each genus 1 component of $C$ mapped by $f$ to a point has at least one special point, and each genus 0 component of $C$ mapped by $f$ to a point has at least three special points; here a \emph{special point} is a marked point or a node.
\end{itemize}
There is a distinguished open substack $\M_{g,n}(X, \beta)$ whose objects are given by stable maps whose source curve $C$ is non-singular. If $X$ is a point then a stable map to $X$ is the same as a marked stable curve, that is $\overline{\M}_{g,n}(\mathrm{pt}, 0) = \overline{\M}_{g,n}$.

Let $\MM_{g,n}$ be the stack of $n$-marked genus $g$ prestable curves. As explained in \cite{behrendgw}, if $X$ is smooth then there exists a natural perfect relative obstruction theory
\begin{equation} \label{eqn: OT stable maps}
	\mathbb{E}_{\overline{\M}_{g,n}(X, \beta)/\MM_{g,n}}^{\bullet} := R^{\bullet}(\pi_{\UU})_{\ast}(f_{\UU}^{\ast} T_X) \to \mathbb{L}^{\bullet}_{\overline{\M}_{g,n}(X, \beta)/\MM_{g,n}},
\end{equation}
where $\mathbb{L}^{\bullet}$ denotes the relative cotangent complex. This gives rise to a virtual fundamental class $[\overline{\M}_{g,n}(X, \beta)]^{\mathrm{vir}}$ of virtual dimension
$$ \mathrm{vdim}(\overline{\M}_{g,n}(X, \beta)) = \int_{\beta} c_1(T_X) + (\dim X - 3)(1 - g) + n. $$

\begin{remark}
	If we fix a closed embedding $i : X \hookrightarrow \PP^b$ such that $i_{\ast} \beta = d'[\ell]$, where $[\ell] \in H_2(\PP^b, \Z)_+$ is the class of a line, then a necessary condition for non-singular stable maps to exist is $2g - 2 + n + 3d' > 0$.
\end{remark}

\section{Stacks of Semistable Sheaves over Stable Maps} \label{section: introducing the stacks}

\subsection{The Stack of Torsion-Free Sheaves over $\overline{\M}_{g,n}(X, \beta)$} Let $X$ be a projective variety and let $\beta \in H_2(X, \Z)_+$ be a curve class. Fix integers $g, n, d, r$, with $g, n \geq 0$, $r \geq 1$ and for which the open substack $\M_{g,n}(X, \beta)$ is non-empty. Let\footnote{The procedure of rigidification, as described in \cite{abramovichcortivistoli}, is also valid for stacks which aren't necessarily algebraic.} $\MM_{g,n,d,r}^{\mathrm{rig}} := \MM_{g,n,d,r} \fatslash \ \G_m$ denote the $\G_m$-rigidification (with respect to the central $\G_m$ in the automorphism groups of the sheaves in $\MM_{g,n,d,r}$) of the stack $\MM_{g,n,d,r}$ of $n$-marked, genus $g$ prestable curves with rank $r$, degree $d$, torsion-free coherent sheaves, and let $\MM_{g,n,d,r}^{\mathrm{simp,rig}}$ be the substack of $\MM_{g,n,d,r}^{\mathrm{rig}}$ parametrising simple sheaves. Set
 $$ \overline{\J}_{g,n,d,r}(X, \beta) := \overline{\M}_{g,n}(X, \beta) \times_{\MM_{g,n}} \MM_{g,n,d,r}^{\mathrm{rig}}, $$
 and
 $$ \overline{\J}^{\mathrm{simp}}_{g,n,d,r}(X, \beta) := \overline{\M}_{g,n}(X, \beta) \times_{\MM_{g,n}} \MM_{g,n,d,r}^{\mathrm{simp,rig}}. $$
 The stack $\overline{\J}_{g,n,d,r}(X, \beta)$ is the stack associated to the CFG $\overline{\J}_{g,n,d,r}(X, \beta)^{\mathrm{pre}}$ whose fibre over a scheme $S$ consists of all objects 
$$(\pi : \mathcal{C} \to S; \sigma_1, \dots, \sigma_n; f : \mathcal{C} \to X; \F),$$ 
where $(\pi : \mathcal{C} \to S; \sigma_1, \dots, \sigma_n; f : \mathcal{C} \to X) \in \overline{\M}_{g,n}(X, \beta)(S)$ is a family of stable maps parametrised by $S$ and where $\F$ is a family of torsion-free sheaves over $\pi : \mathcal{C} \to S$ of relative degree $d$ and rank $r$, and whose morphisms between objects $(\pi : \mathcal{C} \to S; \sigma_1, \dots, \sigma_n; f : \mathcal{C} \to X; \F)$ and $(\pi' : \mathcal{C} \to S; \sigma_1', \dots, \sigma_n'; f' : \mathcal{C} \to X; \F')$ are given by Cartesian diagrams

\begin{equation} \label{eqn: stack morphisms}
	\begin{tikzcd}
		{\mathcal{C}} && {\mathcal{C}'} \\
		\\
		S && {S'}
		\arrow["g"', from=3-1, to=3-3]
		\arrow["{\overline{g}}", from=1-1, to=1-3]
		\arrow["{\pi}"', from=1-1, to=3-1]
		\arrow["{\pi'}", from=1-3, to=3-3]
		\arrow["{\sigma_i}"', curve={height=24pt}, from=1-1, to=3-1]
		\arrow["{\sigma_i'}", curve={height=-24pt}, from=1-3, to=3-3]
		\arrow["\lrcorner"{anchor=center, pos=0.125}, draw=none, from=1-1, to=3-3]
	\end{tikzcd}
\end{equation}
which are compatible with the sections $\sigma_i$ and $\sigma_i'$, together with an equivalence class of isomorphisms $\F \stackrel{\simeq}{\to} \overline{g}^{\ast} \F' \otimes \pi^{\ast} M$ (for some $M \in \Pic S$); isomorphisms $\psi_1 : \F \stackrel{\simeq}{\to} \overline{g}^{\ast} \F' \otimes \pi^{\ast} M_1$ and $\psi_2 : \F \stackrel{\simeq}{\to} \overline{g}^{\ast} \F' \otimes \pi^{\ast} M_2$ are equivalent if there exists an isomorphism $\eta : M_1 \stackrel{\simeq}{\to} M_2$ such that the diagram
\[\begin{tikzcd}
	&& {\overline{g}^{\ast} \F' \otimes \pi^{\ast} M_1} \\
	\F \\
	&& {\overline{g}^{\ast} \F' \otimes \pi^{\ast} M_2}
	\arrow["{\psi_1}", from=2-1, to=1-3]
	\arrow["{\psi_2}"', from=2-1, to=3-3]
	\arrow["{\mathrm{id} \otimes f^{\ast} \eta}", from=1-3, to=3-3]
\end{tikzcd}\]
commutes. If
$$ \overline{\J}ac_{g,n,d,r}(X, \beta) = \overline{\M}_{g,n}(X, \beta) \times_{\MM_{g,n}} \MM_{g,n,d,r} $$ 
(resp.  $\overline{\J}ac^{\mathrm{simp}}_{g,n,d,r}(X, \beta)$) is the stack whose sections coincide with those of $\overline{\J}_{g,n,d,r}(X, \beta)$ (resp. $\overline{\J}^{\mathrm{simp}}_{g,n,d,r}(X, \beta)$) and whose morphisms are given by diagrams \eqref{eqn: stack morphisms} together with isomorphisms $\mathcal{F} \stackrel{\simeq}{\to} \overline{g}^{\ast} \mathcal{F}'$, then $\overline{\J}_{g,n,d,r}(X, \beta)$ is the $\G_m$-rigidification of $\overline{\J}ac_{g,n,d,r}(X, \beta)$. The stack $\overline{\J}_{g,n,d,r}(X, \beta)$ admits a universal curve and a universal morphism to $X$, obtained by pulling back the universal family over $\overline{\M}_{g,n}(X, \beta)$. The stack $\overline{\J}ac_{g,n,d,r}(X, \beta)$ admits a universal sheaf over this universal curve. 

The stack $\overline{\J}_{g,n,d,r}(X, \beta)$ admits natural morphisms to the stacks $\overline{\M}_{g,n}(X, \beta)$ and $\MM_{g,n,d,r}^{\mathrm{rig}}$, obtained by forgetting respectively the sheaves and the morphisms to $X$.

\begin{prop} \label{prop: algebraicity result}
	The stack $\overline{\J}^{\mathrm{simp}}_{g,n,d,r}(X, \beta)$ is Deligne-Mumford and the natural forgetful morphism to $\overline{\M}_{g,n}(X, \beta)$ is representable and locally of finite presentation. If $X$ is smooth then any open Deligne-Mumford substack of $\overline{\J}^{\mathrm{simp}}_{g,n,d,r}(X, \beta)$ admits a natural perfect relative obstruction theory over $\MM_{g,n,d,r}^{\mathrm{simp, rig}}$.
\end{prop}

\begin{proof}
	We first show that the stack $\MM^{\mathrm{simp, rig}}_{g,n,d,r}$ is algebraic and locally of finite presentation over $\MM_{g,n}$. Consider a morphism from a scheme $S \to \MM_{g,n}$ and let
	$$ \MM_S = \MM^{\mathrm{simp, rig}}_{g,n,d,r} \times_{\MM_{g,n}} S. $$
	If $f : \mathcal{C} \to S$ is the family of prestable curves associated to $S \to \MM_{g,n}$ then $\MM_S$ is equivalent to the étale sheafification $(\mathbb{M}_S)_{\mathrm{\acute{e}t}}$ of the functor $\mathbb{M}_S$ assigning to an $S$-scheme $T$ the set of $T$-flat simple rank $r$, degree $d$, torsion-free coherent sheaves on $\mathcal{C}_T$. By \cite[Theorem 7.4]{altmankleiman}, $(\mathbb{M}_S)_{\mathrm{\acute{e}t}}$ is an algebraic space which is locally finitely presented over $S$, and so the natural morphism $\MM^{\mathrm{simp, rig}}_{g,n,d,r} \to \MM_{g,n}$ is representable and locally of finite presentation. As $\MM_{g,n}$ is well-known to be algebraic (see for instance \cite{hall}) it follows that $\MM^{\mathrm{simp, rig}}_{g,n,d,r}$ is algebraic.
	
	Base-changing to $\overline{\M}_{g,n}(X, \beta)$ shows that the natural morphism $c : \overline{\J}^{\mathrm{simp}}_{g,n,d,r}(X, \beta) \to \overline{\M}_{g,n}(X, \beta)$ is also representable and locally of finite presentation. As $\overline{\M}_{g,n}(X, \beta)$ is Deligne-Mumford, the same is therefore true for $\overline{\J}^{\mathrm{simp}}_{g,n,d,r}(X, \beta)$. In the case where $X$ is smooth, the pullback of the perfect relative obstruction theory \eqref{eqn: OT stable maps} along $c$ defines a perfect relative obstruction theory
	\begin{equation}
		\mathbb{E}_{\overline{\J}^{\mathrm{simp}}_{g,n,d,r}(X, \beta)/\MM^{\mathrm{simp, rig}}_{g,n,d,r}}^{\bullet} := c^{\ast} \mathbb{E}_{\overline{\M}_{g,n}(X, \beta)/\MM_{g,n}}^{\bullet} \to \mathbb{L}^{\bullet}_{\overline{\J}^{\mathrm{simp}}_{g,n,d,r}(X, \beta)/\MM^{\mathrm{simp, rig}}_{g,n,d,r}},
	\end{equation}
	which restricts to give a perfect relative obstruction theory for any open Deligne-Mumford substack of $\overline{\J}^{\mathrm{simp}}_{g,n,d,r}(X, \beta)$.
\end{proof}

\begin{remark}
	As $\overline{\J}ac_{g,n,d,r}(X, \beta)$ is a $\G_m$-gerbe over the rigidified stack $\overline{\J}_{g,n,d,r}(X, \beta)$, it follows from Proposition \ref{prop: algebraicity result} that the stack $\overline{\J}ac^{\mathrm{simp}}_{g,n,d,r}(X, \beta)$ is also algebraic.
\end{remark}	

\subsection{Substacks Defined by Polarisations} \label{section: substacks defined by polarisations} Let $\pi_{\UU} : \UU\overline{\M}_{g,n}(X, \beta) \to \overline{\M}_{g,n}(X, \beta)$ be a relatively ample $\Q$-invertible sheaf $\LL_{\UU}$ on the universal curve $\UU\overline{\M}_{g,n}(X, \beta)$. For any family of coherent sheaves $\F$ over a family of maps $(\pi : \mathcal{C} \to S; \sigma_1, \dots, \sigma_n; f : \mathcal{C} \to X)$ corresponding to a morphism $S \to \overline{\M}_{g,n}$, by pulling back $\LL_{\UU}$ we obtain a relatively ample $\Q$-invertible sheaf $\LL_S$ over $\mathcal{C}$.

\begin{defn}
	We say $\F$ is $\LL_{\UU}$-\emph{(semi)stable} if for each geometric point $s \in S$, the sheaf $\F_s = \F \times_{S} \spec \C$ is Gieseker (semi)stable with respect to the ample $\Q$-invertible sheaf $\LL_s$.
\end{defn}

The conditions of being stable/semistable are both open in flat families, so there are open substacks
$$ \overline{\J}ac_{g,n,d,r}^{(s)s}(X, \beta)(\LL_{\UU}) \subset \overline{\J}ac_{g,n,d,r}(X, \beta) \text{   and   } \overline{\J}_{g,n,d,r}^{(s)s}(X, \beta)(\LL_{\UU}) \subset \overline{\J}_{g,n,d,r}(X, \beta) $$
parametrising objects whose sheaves are $\LL_{\UU}$-(semi)stable. As Gieseker stable sheaves are always simple (cf. \cite{huybrechts_lehn} Corollary 1.2.8), we have open inclusions
$$ \overline{\J}ac_{g,n,d,r}^{s}(X, \beta)(\LL_{\UU}) \subset \overline{\J}ac_{g,n,d,r}^{\mathrm{simp}}(X, \beta) \text{   and   } \overline{\J}_{g,n,d,r}^{s}(X, \beta)(\LL_{\UU}) \subset \overline{\J}_{g,n,d,r}^{\mathrm{simp}}(X, \beta). $$
Applying Proposition \ref{prop: algebraicity result}, the stack $\overline{\J}ac_{g,n,d,r}^{s}(X, \beta)(\LL_{\UU})$ is algebraic, and the rigidified stack $\overline{\J}_{g,n,d,r}^{s}(X, \beta)(\LL_{\UU})$ is Deligne-Mumford and admits a perfect relative obstruction theory over $\MM_{g,n,d,r}^{\mathrm{simp,rig}}$.

If $(C; x_1, \dots, x_n; f : C \to X)$ is a stable map in $\overline{\M}_{g,n}(X, \beta)$ with $C$ irreducible, the slope function \eqref{eqn: multi slope stability nodal curve} reduces to the usual Mumford slope $\mu(E) = \deg(E)/\mathrm{rk}(E)$. If $C$ is additionally non-singular then all torsion-free sheaves on $C$ are locally free. This implies that $\overline{\J}_{g,n,d,r}^{ss}(X, \beta)(\LL_{\UU})$ contains all slope semistable vector bundles over stable maps whose source curve is non-singular. In particular, if $\M_{g,n}(X, \beta)$ is non-empty then the stack $\overline{\J}_{g,n,d,r}^{ss}(X, \beta)(\LL_{\UU})$ is also non-empty. 

\begin{remark}
	It follows from Inequality \ref{eqn: multi slope stability nodal curve} that replacing $\LL_{\UU}$ with a positive integral power of $\LL_{\UU}$ does not alter the stability condition. As such, we are free to assume as and when necessary that $\LL_{\UU}$ is a genuine relatively very ample invertible sheaf.
\end{remark} 

\begin{remark}
	Proposition \ref{prop: algebraicity result} does not address the issue as to whether $\overline{\J}ac_{g,n,d,r}^{ss}(X, \beta)(\LL_{\UU})$ and $\overline{\J}ac_{g,n,d,r}^{ss}(X, \beta)(\LL_{\UU})$ are algebraic in general. The algebraicity of these stacks will follow as a consequence of Theorem \ref{thm: good moduli spaces as GIT quotients}.
\end{remark}

\begin{remark}
	Given an invertible sheaf $\mathcal{N}$ on the universal curve $\UU\overline{\M}_{g,n}(X, \beta)$, one may consider the stacks $\overline{\J}ac_{g,n,d,r}^{ss}(X, \beta)(\LL, \mathcal{N})$ and $\overline{\J}_{g,n,d,r}^{ss}(X, \beta)(\LL, \mathcal{N})$ parametrising $(\mathcal{L}, \mathcal{N})$-twisted semistable sheaves of degree $d$ and uniform rank $r$ over stable maps in $\overline{\mathcal{M}}_{g,n}(X, \beta)$ (that is, we impose that for each geometric point $s \in S$, $F_s \otimes \mathcal{N}_s$ is Gieseker semistable with respect to $\mathcal{L}_s$). This does not widen the class of stacks under consideration, since twisting by $\mathcal{N}$ yields isomorphisms
	$$ \overline{\J}ac_{g,n,d,r}^{ss}(X, \beta)(\LL, \mathcal{N}) \cong \overline{\J}ac_{g,n,d,r}^{ss}(X, \beta)(\LL, \OO) $$ 
	and 
	$$ \overline{\J}_{g,n,d,r}^{ss}(X, \beta)(\LL, \mathcal{N}) \cong \overline{\J}_{g,n,d,r}^{ss}(X, \beta)(\LL, \OO). $$
	In particular, the compactified Jacobians considered by Kass--Pagani \cite{kasspagani} and Melo \cite{melocuj} fit into the framework considered in this paper, since the stability conditions considered by these authors can always be realised as twisted stability conditions.
\end{remark}

\subsection{Basic Properties of the Stacks} We summarise some of the basic properties of the stacks $\overline{\J}_{g,n,d,r}^{ss}(X, \beta)(\LL_{\UU})$.

\begin{enumerate}
	\item The stack $\overline{\J}_{g,n,d,r}^{ss}(X, \beta)(\LL_{\UU})$ is algebraic and of finite presentation over $\C$; this follows immediately from Theorem \ref{thm: Theorem A}.
	\item The stack $\overline{\J}_{g,n,d,r}^{ss}(X, \beta)(\LL_{\UU})$ is universally closed. This will follow once we have established the projectivity of the good moduli space $\overline{J}_{g,n,d,r}^{ss}(X, \beta)(\LL_{\UU})$ of this stack, as good moduli space morphisms are universally closed.
	\item If $\overline{\J}_{g,n,d,r}^{s}(X, \beta)(\LL_{\UU}) = \overline{\J}_{g,n,d,r}^{ss}(X, \beta)(\LL_{\UU})$ then the stack is Deligne-Mumford and the natural forgetful morphism to $\overline{\M}_{g,n}(X, \beta)$ is representable. If $X$ is smooth then the stack admits a natural relative perfect obstruction theory over the stack $\MM_{g,n,d,r}^{\mathrm{simp,rig}}$; this follows from Proposition \ref{prop: algebraicity result}.
	\item If $\overline{\J}_{g,n,d,r}^{s}(X, \beta)(\LL_{\UU}) = \overline{\J}_{g,n,d,r}^{ss}(X, \beta)(\LL_{\UU})$ then the stack is proper. One way to see this is to notice that every object of $\overline{\J}_{g,n,d,r}^{ss}(X, \beta)(\LL_{\UU})$ has a finite automorphism group scheme, or in other words the inertia stack of $\overline{\J}_{g,n,d,r}^{ss}(X, \beta)(\LL_{\UU})$ is finite. It then follows from \cite{conradkm} that the coarse moduli space morphism $\overline{\J}_{g,n,d,r}^{ss}(X, \beta)(\LL_{\UU}) \to \overline{J}_{g,n,d,r}^{ss}(X, \beta)(\LL_{\UU})$ is proper, so the properness of the stack $\overline{\J}_{g,n,d,r}^{ss}(X, \beta)(\LL_{\UU})$ is a consequence of the projectivity of $\overline{J}_{g,n,d,r}^{ss}(X, \beta)(\LL_{\UU})$.
\end{enumerate}

In the case $r = 1$, we have the following additional properties:

\begin{enumerate}
	\setcounter{enumi}{4}
	\item If $n \geq 1$ then the universal sheaf on $\overline{\J}ac_{g,n,d,1}^{ss}(X, \beta)(\LL_{\UU})$ descends to the rigidification $\overline{\J}_{g,n,d,1}^{ss}(X, \beta)(\LL_{\UU})$; this follows from either \cite[Lemma 3.35]{kptheta} or \cite[Proposition 3.5]{melocuj}.
	\item If the stack of stable maps $\overline{\M}_{g,n}(X, \beta)$ is smooth then the stack $\overline{\J}_{g,n,d,1}^{ss}(X, \beta)(\LL_{\UU})$ is also smooth; the argument given in the proof of \cite[Proposition 3.7]{melocuj} carries over. If in addition the stack $\overline{\M}_{g,n}(X, \beta)$ is irreducible and if the substack $\M_{g,n}(X, \beta)$ is dense in $\overline{\M}_{g,n}(X, \beta)$ then the stack $\overline{\J}_{g,n,d,1}^{ss}(X, \beta)(\LL_{\UU})$ is irreducible and of dimension $\dim \overline{\M}_{g,n}(X, \beta) + g$, as in this case the open substack parametrising line bundles over stable maps in $\M_{g,n}(X, \beta)$ is smooth and dense. This applies for instance when the stack of stable maps $\overline{\M}_{g,n}(X, \beta)$ is taken to be $\overline{\M}_{g,n}$ or $\overline{\M}_{0,n}(\PP^b, d')$.
\end{enumerate}

\section{GIT Constructions of Moduli Spaces of Stable Maps and of Multi-Gieseker Semistable Sheaves} \label{section: BS GRT constructions}

The GIT construction of the moduli spaces appearing in the statement of Theorem \ref{thm: Theorem A} relies on two existing GIT constructions, the first by Baldwin-Swinarski \cite{baldwinswinarski} and the second by Greb-Ross-Toma \cite{grebrosstoma} \cite{grebrosstomamaster}. In preparation for the following section, where these constructions are combined to construct the moduli spaces $\overline{J}_{g,n,d,r}^{ss}(X, \beta)(\LL_{\UU})$, we give a sketch of both of these constructions, and explain how the latter construction may be extended to the relative setting. To avoid the notation from becoming too complicated, we sketch the Greb-Ross-Toma construction in a general setting, largely following the notation used in \cite{grebrosstoma} and \cite{grebrosstomamaster}.

\subsection{The Construction of Baldwin and Swinarski} \label{section: BS construction} Let $X$ be a projective variety, and fix the discrete invariants $g$, $n$ and $\beta$. For simplicity assume there is at least one marked point (the construction for $n = 0$ marked points involves only minor changes in notation). Fix a closed embedding $i : X \hookrightarrow \PP^b$ such that $i_{\ast} \beta = d'[\ell]$, where $[\ell] \in H_2(\PP^b, \Z)_+$ is the class of a line. Suppose $(C; x_1, \dots, x_n; f)$ is a stable map to $X$ of class $\beta$ (which we may also consider as a stable map to $\PP^b$ of class $d'[\ell]$). Set\footnote{$a = 10$ is the smallest choice of $a$ for which the statement of \cite[Theorem 5.21]{baldwinswinarski} holds for any choice of valid discrete invariants $g, n, d'$; we refer the reader to the remark following \emph{loc. cit.}}
$$ e = \deg((\omega_C(x_1 + \cdots + x_n) \otimes f^{\ast} \OO_{\PP^b}(3))^{10}) = 10(2g - 2 + n + 3d'). $$
The line bundle $(\omega_C(x_1 + \cdots + x_n) \otimes f^{\ast} \OO_{\PP^b}(3))^{10}$ is non-special and very ample, and corresponds to an embedding of $C$ inside $\PP(H^0(C, (\omega_C(x_1 + \cdots + x_n) \otimes f^{\ast} \OO_{\PP^b}(3))^{10}))$. Let 
$$ W = \C^{e-g+1}. $$
Choose an isomorphism $W \cong H^0(C, (\omega_C(x_1 + \cdots + x_n) \otimes f^{\ast} \OO_{\PP^b}(3))^{10})$; this yields an induced embedding $C \hookrightarrow \PP(W)$. The graph of $f$ gives in turn a closed embedding $C \hookrightarrow \PP(W) \times \PP^b$. After appending the markings $(x_i, f(x_i))$, we have an associated point of the scheme
$$ \HH = \mathrm{Hilb}(\PP(W) \times \PP^b, P_0) \times (\PP(W) \times \PP^b)^{\times n}. $$
Here $\mathrm{Hilb}(\PP(W) \times \PP^b, P_0)$ is the Hilbert scheme parametrising curves $C' \subset \PP(W) \times \PP^b$ with Hilbert polynomial
$$ P_0(m_W, m_b) = \chi(\OO_{\PP(W)}(m_W) \otimes \OO_{\PP^b}(m_b) \otimes \OO_{C'}) = em_W + dm_b - g + 1. $$
Let $(\phi_{\UU} : \UU\HH \to \HH; \tau_1, \dots, \tau_n)$ denote the universal family over $\HH$, given by taking the product of the universal family over $\mathrm{Hilb}(\PP(W) \times \PP^b, P_0)$ with $n$ copies of the identity map on $\PP(W) \times \PP^b$. Let $\HH' \subset \HH$ denote the closed subscheme consisting of all points $(h; x_1, \dots, x_n) \in \HH$ where the marked points $x_i$ lie on the curve $\UU\HH_h \subset \PP(W) \times \PP^b$. As explained in \cite[Sections 2.3 and 5.1]{fultonpandharipande}, there exists a locally closed subscheme $I = I_{X, \beta} \subset \HH'$ consisting of all points $(h; x_1, \dots, x_n) \in \HH'$ such that the following conditions are satisfied:\footnote{The subscheme $\HH'$ is what Baldwin and Swinarski call $I$ and the subscheme $I$ is what Baldwin and Swinarski call $J$; we have elected to change the notation so as not to confuse notation with that for the moduli spaces of the stacks $\overline{\J}_{g,n,d,r}^{ss}(X, \beta)(\LL_{\UU})$.}
\begin{enumerate}
	\item $(\UU\HH_h; x_1, \dots, x_n)$ is a prestable marked curve;
	\item the projection $\UU\HH_h \to \PP(W)$ is a non-degenerate embedding;
	\item the projection $p_b : \UU\HH_h \to \PP^b$ factors through the inclusion $i : X \hookrightarrow \PP^b$; 
	\item there is an equality of homology classes $(p_b)_{\ast}[\UU\HH_h] = \beta \in H_2(X, \Z)_+$; and
	\item there exists an isomorphism of line bundles on the curve $\UU\HH_h$
	$$ \OO_{\PP(W)}(1) \otimes \OO_{\PP^b}(1) \otimes \OO_{\UU\HH_h} \cong (\omega_{\UU\HH_h}(x_1 + \cdots + x_n))^{10} \otimes \OO_{\PP^b}(31) \otimes \OO_{\UU\HH_h} $$  
	(more formally, if $S \to \HH'$ is a morphism from a scheme, then this morphism factors through $I$ only if the line bundles $\OO_{\PP(W)}(1) \otimes \OO_{\PP^b}(1) \otimes \OO_{\UU\HH'_S}$ and $(\omega_{\UU\HH'_S/S}(\tau_1(S) + \cdots + \tau_n(S)))^{10} \otimes \OO_{\PP^b}(31) \otimes \OO_{\UU\HH'_S}$ differ by the pullback of a line bundle on $S$).
\end{enumerate}
Let $\overline{I}$ denote the closure of $I$ in $\HH'$, and consider the restriction of the universal family to $I$, denoted as $(\phi_{\UU} : \UU I \to I; \tau_1, \dots, \tau_n)$. There is a universal morphism $\UU I \to X$ given by the composition of the closed embedding 
$$\UU I \hookrightarrow I \times \PP(W) \times \PP^b \times (\PP(W) \times \PP^b)^{\times n}$$ 
with the projection onto the first $\PP^b$ factor. The restriction of this morphism over each point of $I$ defines a stable map to $X$ of class $\beta$. It follows that the scheme $I$ admits a natural forgetful morphism to $\overline{\M}_{g,n}(X, \beta)$, given by forgetting the embedding of the stable map inside $\PP(W) \times \PP^b$. 

Letting $GL(W)$ act in the usual way on $\PP(W)$ and act trivially on $\PP^b$, the corresponding diagonal action on $\PP(W) \times \PP^b$ gives rise to an induced action of $GL(W)$ on $\HH$, under which the subschemes $\HH'$, $I$ and $\overline{I}$ are all invariant. This action descends to an action of $PGL(W)$. The family $(\phi_{\UU} : \UU I \to I; \tau_1, \dots, \tau_n)$ has the local universal property for $\overline{\mathcal{M}}_{g,n}(X, \beta)$, and two points $i_1, i_2 \in I$ correspond to isomorphic stable maps if and only if $i_1$ and $i_2$ lie in the same $GL(W)$-orbit (cf. \cite[Proposition 3.4]{baldwinswinarski}).

For sufficiently large positive integers $m_W$ and $m_b$, consider the Grothendieck embedding $$\mathrm{Hilb}(\PP(W) \times \PP^b, P_0) \hookrightarrow \PP(\Lambda^{P_0(m_W, m_b)} Z_{m_W, m_b}),$$ 
where
$$ Z_{m_W,m_b} = H^0(\PP(W) \times \PP^b, \OO_{\PP(W)}(m_W) \otimes \OO_{\PP^b}(m_b)). $$
We extend this to a closed embedding
$$ \HH = \mathrm{Hilb}(\PP(W) \times \PP^b, P_0) \times (\PP(W) \times \PP^b)^{\times n} \hookrightarrow \PP(\Lambda^{P_0(m_W, m_b)} Z_{m_W, m_b}) \times (\PP(W) \times \PP^b)^{\times n} $$
by taking the identity on $(\PP(W) \times \PP^b)^{\times n}$. After choosing a positive integer $m_{\mathrm{pts}}$, applying appropriate Veronese and Segre embeddings gives rise to a $GL(W)$-equivariant closed embedding of $\HH$ inside a single projective space:
\begin{equation} \label{eqn: embedding 1}
	\theta_{m_W, m_b, m_{\mathrm{pts}}} : \HH \hookrightarrow \PP \left( \Lambda^{P_0(m_W, m_b)} Z_{m_W, m_b} \otimes \bigotimes_{i=1}^n ( \sym^{m_{\mathrm{pts}}} W \otimes \sym^{m_{\mathrm{pts}}} \C^{b+1} ) \right). 
\end{equation}
Pulling back the $\OO(1)$ along this embedding gives a very ample linearisation $L_{m_W, m_b, m_{\mathrm{pts}}}$ for the action of $GL(W)$ on both $\HH$ and the closed subscheme $\overline{I}$. We may now state the main result of Baldwin and Swinarski.

\begin{theorem}[\cite{baldwinswinarski}, Corollary 6.2] \label{thm: baldwin swinarski}
	Suppose $(m_W, m_b, m_{\mathrm{pts}})$ is a triple of positive integers satisfying the inequalities in the statement of Theorem 6.1 of loc. cit. (such triples always exist). Then for the induced action of $SL(W)$ on $\overline{I}$ there are equalities 
	$$\overline{I}^{ss}(L_{m_W, m_b, m_{\mathrm{pts}}}) = \overline{I}^{s}(L_{m_W, m_b, m_{\mathrm{pts}}}) = I.$$ 
	Moreover $\overline{M}_{g,n}(X, \beta)$ is isomorphic to the GIT quotient $\overline{I} \sslash_{L_{m_W, m_b, m_{\mathrm{pts}}}} SL(W)$. In particular $\overline{M}_{g,n}(X, \beta)$ is a geometric quotient of $I$.
\end{theorem}

\subsection{Quiver Representation Stability} We now turn to the construction of Greb-Ross-Toma \cite{grebrosstoma} \cite{grebrosstomamaster}.

Given $k \in \mathbb{N}$, consider the labelled quiver
$$ \mathcal{Q} = (\mathcal{Q}_0, \mathcal{Q}_1, h, t : \mathcal{Q}_1 \to \mathcal{Q}_0, H : \mathcal{Q}_1 \to \mathbf{Vect}^{\mathrm{fd}}_{\C}) $$
with vertex set
$$ \mathcal{Q}_0 = \{v_i, w_j : i, j = 1, \dots, k\}, $$
arrow set
$$ \mathcal{Q}_1 = \{\alpha_{ij} : i, j = 1, \dots, k\}, $$
and with heads and tails given by
$$ h(\alpha_{ij}) = v_i, \quad t(\alpha_{ij}) = w_j $$
(the functor $H$ will be specified momentarily). Let $H_{ij}$ be the vector space $H(\alpha_{ij})$. 

\begin{defn}
    A \emph{representation} of the labelled quiver $\mathcal{Q}$ consists of the data 
    $$\psi = (\{V_i\}_{i=1}^k, \{W_j\}_{j=1}^k, \{\psi_{ij}\}_{i,j = 1}^k),$$
    where the $V_i$ and $W_j$ are finite-dimensional $\C$-vector spaces and where each $\psi_{ij}$ is a linear map $V_i \otimes H_{ij} \to W_j$. 
    
    Morphisms of representations of $\mathcal{Q}$ are defined in analogy with morphisms of ordinary quiver representations (cf. \cite{kingquiverreps}); in particular, a subrepresentation $\psi'$ of $\psi$ is specified by choices of subspaces $V_i' \subset V_i$ and $W_j' \subset W_j$, such that each $\psi_{ij}$ sends $V_i' \otimes H_{ij}$ to $W_j'$. 
    
    The \emph{dimension vector} of $\psi$ is the vector $\underline{d} = (\dim V_1, \dim W_1, \dots, \dim V_k, \dim W_k)$.
\end{defn}

Write the dimension vector of $M = \bigoplus_j V_j \oplus W_j$ as $\underline{d} = (d_{11}, d_{12}, \dots, d_{k1}, d_{k2})$. Given a tuple $\sigma = (\sigma_1, \dots, \sigma_k)$ of non-negative rational numbers, not all of which are zero, define $\theta_{\sigma} = (\theta_{11}, \theta_{12}, \dots, \theta_{k1}, \theta_{k2})$ by setting
$$ \theta_{j1} = \frac{\sigma_j}{\sum_i \sigma_i d_{i1}}, \quad \theta_{j2} = - \frac{\sigma_j}{\sum_i \sigma_i d_{i2}}. $$
Then, for any representation $M' = \bigoplus_j V_j' \oplus W_j'$, define
$$ \theta_{\sigma}(M') = \sum_{j=1}^k (\theta_{j1} \dim V_j' + \theta_{j2} \dim W_j'). $$

\begin{defn}
	Let $M$ be a $\mathcal{Q}$-representation with dimension vector $\underline{d}$. We say that $M$ is \emph{(semi)stable with respect to $\sigma$} if for all proper subrepresentations $M' \subset M$, we have $\theta_{\sigma}(M') < (\leq) \ 0$.
\end{defn}

Analogously to the case of ordinary quiver representations, every $\sigma$-semistable $\mathcal{Q}$-representation admits a Jordan-H\"older filtration, and hence an \emph{associated graded} object $\mathrm{gr}_{\sigma}(M)$. This gives a notion of \emph{S-equivalence} for $\sigma$-semistable representations; that is, two $\sigma$-semistable representations $M$ and $M'$ are equivalent if $\mathrm{gr}_{\sigma}(M) \cong \mathrm{gr}_{\sigma}(M')$. 

Fixing a dimension vector $\underline{d}$, let $R$ be the vector space
$$ R = \mathrm{Rep}(\mathcal{Q},d) = \bigoplus_{i, j = 1}^k \mathrm{Hom}(\C^{d_{i1}} \otimes H_{ij}, \C^{d_{j2}}). $$
Each point of $R$ corresponds to a representation of dimension vector $\underline{d}$. There is a natural linear left action on $R$ by
$$ \hat{G} = \prod_{j=1}^k (GL_{d_{j1}}(\C) \times GL_{d_{j2}}(\C)), $$
where $\hat{G}$ acts by base change automorphisms. Let 
$$ \Delta = \{ (t \cdot \mathrm{id}, \dots, t \cdot \mathrm{id}) \in \hat{G} : t \in \G_m \} $$ 
be the kernel of the representation $\hat{G} \to GL(R)$ and let $G = \hat{G} / \Delta$. Given an integral vector $\theta = (\theta_{j1}, \theta_{j2}, \dots, \theta_{k1}, \theta_{k2}) \in \Z^{2k}$, let $\chi_{\theta}$ be the character of $\hat{G}$ given by
$$ \chi_{\theta} : g \mapsto \prod_{j=1}^k (\det(g_{j1})^{-\theta_{j1}} \cdot \det(g_{j2})^{-\theta_{j2}}). $$
Assume $\underline{d}$ and $\theta$ are chosen so that $\sum_j (\theta_{j1} d_{j1} + \theta_{j2} d_{j2}) = 0$, so that we may (and do) view $\chi_{\theta}$ as a character of $G$. Let $R^{\theta-(s)s} \subset R$ denote the GIT (semi)stable loci for the action of $G$ on $R$ with the linearisation $\OO_R(\chi_{\theta})$, the trivial bundle on $R$ twisted by the character $\chi_{\theta}$, and let $p : R^{\theta-ss} \to R \sslash_{\theta} G$ denote the resulting good quotient. In this setting, the results of King \cite{kingquiverreps} are as follows.

\begin{prop}[\cite{grebrosstoma} Theorem 5.5] \label{prop: quiver GIT}
	Suppose $\theta = c\theta_{\sigma}$ for some $\sigma$ and some positive integer $c$, where $c$ is chosen such that $\theta$ is integral.
	\begin{enumerate}
		\item A point of $R$ corresponding to a $\mathcal{Q}$-representation $M$ is GIT (semi)stable with respect to $\theta$ if and only if $M$ is (semi)stable with respect to $\sigma$.
		\item If $r_1, r_2 \in R^{\theta-ss}$ are points corresponding to representations $M_1$ and $M_2$ respectively, then $p(r_1) = p(r_2)$ if and only if $M_1$ and $M_2$ are S-equivalent $\sigma$-semistable representations.
	\end{enumerate}
\end{prop}

\subsection{The Functorial Approach to Moduli Spaces of Sheaves} \label{section: functorial approach} Let $X$ be a projective scheme, and fix a stability parameter $\sigma = (\underline{L}, \sigma_1, \dots, \sigma_k)$ on $X$. Fix a topological type $\tau$ of coherent sheaves on $X$ defined with respect to $\underline{L}$, and let $P_j(t)$ be the corresponding Hilbert polynomials. Set $H_{ij} = H^0(X, L_i^{-m_1} \otimes L_j^{m_2}) = \mathrm{Hom}(L_j^{-m_2}, L_i^{-m_1})$, where $m_2 > m_1 > 0$ are integers to be determined. Let
$$ T = \bigoplus_{j=1}^k L_j^{-m_1} \oplus L_j^{-m_2}, $$
and consider the algebra
$$ A = L \oplus H \subset \mathrm{End}_X(T), $$
where $L$ is generated by the projection operators onto the summands $L_i^{-m_1}$ and $L_j^{-m_2}$ of $T$ and where $H = \bigoplus_{i, j = 1}^k H_{ij}$. $T$ is a left $A$-module and $H$ is an $L$-bimodule. The category of representations of $\mathcal{Q}$ is equivalent to the category of finite-dimensional, right $A$-modules $M$. From now on we identify these two categories.

Given a coherent sheaf $E$ on $X$, the vector space $\mathrm{Hom}(T,E)$ has a natural $A$-module structure, given by the decomposition
$$ \mathrm{Hom}(T,E) = \bigoplus_{j=1}^k H^0(E \otimes L_j^{m_1}) \oplus H^0(E \otimes L_j^{m_2}) $$
together with the multiplication maps $H^0(E \otimes L_i^{m_1}) \otimes H_{ij} \to H^0(E \otimes L_j^{m_2})$. This defines a functor
$$ \mathrm{Hom}(T,-) : \mathbf{Coh}(X) \to \mathbf{mod}(A). $$
The functor $\mathrm{Hom}(T,-)$ admits a left adjoint, denoted $- \otimes_A T$. If $\mu : H \otimes_L T \to T$ is the left $A$-module structure map, then $M \otimes_A T$ is the cokernel of the map
$$ 1 \otimes \mu - \alpha \otimes 1 : M \otimes_L H \otimes_L T \to M \otimes_L T. $$
If $m_2 \gg m_1$, by \cite[Theorem 3.4]{grebrosstoma} the functor $\mathrm{Hom}(T,-)$ is fully faithful on the full subcategory of $(m_1, \underline{L})$-regular coherent sheaves on $X$ of topological type $\tau$.

Now consider the case where $X$ is projective over a quasi-projective scheme $Y$. Fix a projective morphism $\pi : X \to Y$,\footnote{In Section \ref{section: carrying out the construction} this projective morphism will be taken to be the universal family $\UU I \to I$.} and replace $\underline{L}$ with $\underline{\LL}$, a tuple of $\pi$-very ample invertible sheaves $\LL_1, \dots, \LL_k$. Continue to fix a topological type $\tau$ of $Y$-flat sheaves on $X$ (this time defined with respect to the $\LL_i$); let $P_1(t), \dots, P_k(t)$ be the associated Hilbert polynomials. Replace $H_{ij}$ with the sheaf $\pi_{\ast}(\LL_i^{-m_1} \otimes \LL_j^{m_2})$, so that $A$ is now a sheaf of $\OO_Y$-algebras, acting naturally on $T$.

\begin{defn}
    Let $g : S \to Y$ be a scheme over $Y$. A \emph{flat family of $A$-modules over $S$} consists of a sheaf $\M$ of right modules over the sheaf of algebras $A_S = g^{\ast}A$, which is locally free of finite rank as an $\OO_S$-module.
\end{defn}

If $S$ is a $Y$-scheme and if $\E$ is a coherent sheaf on $X_S = X \times_Y S$, set
$$ \Hom(T, \E) := (\pi_S)_{\ast}(\Hom_{X_S}(T_S, \E)) \in \mathbf{Coh}(S). $$
This is naturally a coherent sheaf of right $A_S$-modules. The functor $- \otimes_A T$ extends to give a left adjoint of $\Hom(T, -)$. 

\begin{remark}
    In Proposition \ref{prop: embedding 1} and in all of the statements that follow, $m_2 - m_1$ can and should be chosen large enough such that each of the sheaves $H_{ij} = \pi_{\ast}(\LL_i^{-m_1} \otimes \LL_j^{m_2})$ are locally free and the formation of the pushforwards all commute with base change.
\end{remark}

\begin{prop}[\cite{grebrosstoma} Propositions 5.8, 5.9, \cite{alvarezconsulking} Propositions 4.1, 4.2 in the relative setting] \label{prop: embedding 1} \label{prop: embedding 2}
    Let $m_1$ be a natural number. Then for $m_2 \gg m_1$, the following holds. 
    \begin{enumerate}
        \item For any $Y$-scheme $S$, the functor $\Hom(T,-)$ is a fully faithful functor from the full subcategory of $S$-flat, $(m_1, \underline{\LL})$-regular coherent sheaves on $X_S$ of topological type $\tau$ to the full subcategory of $\mathbf{mod}(A_S)$ consisting of flat families of $A$-modules over $S$.
        \item If $B$ is any $Y$-scheme and if $\M$ is any flat family of $A$-modules over $B$ of dimension vector
        $$ \underline{d} = (P_1(m_1), P_1(m_2), \dots, P_k(m_1), P_k(m_2)), $$
        then there exists a unique locally closed subscheme $\iota : B_{\tau}^{[reg]} \hookrightarrow B$ with the following properties:
        \begin{enumerate}
            \item $\iota^{\ast} \M \otimes_A T$ is a $B_{\tau}^{[reg]}$-flat, $B_{\tau}^{[reg]}$-finitely presented family of $(m_1, \underline{\LL})$-regular sheaves on $X$ of topological type $\tau$, and the unit map $\iota^{\ast} \M \to \Hom(T, \iota^{\ast} \M \otimes_A T)$ is an isomorphism.
            \item If $g : S \to B$ is such that there exists an $S$-flat, $S$-finitely presented family $E$ of $(m_1, \underline{\LL})$-regular sheaves on $X$ of topological type $\tau$ and an isomorphism $g^{\ast} \M \cong \Hom(T,E)$, then $g$ factors through $\iota : B_{\tau}^{[reg]} \hookrightarrow B$ and $E \cong g^{\ast} \M \otimes_A T$.
        \end{enumerate}
    \end{enumerate}
\end{prop}

\begin{proof}
	This follows from making minor modifications to the proofs of \cite[Propositions 5.8 and 5.9]{grebrosstoma}, these modifications being analogous to those indicated in \cite[Section 6.5]{alvarezconsulking}.
\end{proof}

\subsection{Comparison of Semistability - Positive Case} \label{section: comparison of semistability}

Given a coherent sheaf $E$ of topological type $\tau$ on a projective scheme $X$, we have a notion of multi-Gieseker stability for the sheaf $E$ and a notion of stability for the $A$-module $\mathrm{Hom}(T,E)$, both with respect to a common stability parameter $\sigma = (\underline{L}, \sigma_1, \dots, \sigma_k)$. Provided $\sigma$ is bounded, the following result relates these two notions.

\begin{theorem}[\cite{alvarezconsulking} Section 5, \cite{grebrosstoma} Theorem 8.1] \label{thm: comparison theorem}
	Let $\pi : X \to Y$ be a projective morphism, where $Y$ is a quasi-projective scheme. Let $\underline{\LL} = (\LL_1, \dots, \LL_k)$ be a tuple of $\pi$-very ample invertible sheaves and let $\tau$ be a topological type defined with respect to $\underline{\LL}$. Let $\sigma = (\underline{\LL}, \sigma_1, \dots, \sigma_k)$ be a stability parameter which is bounded with respect to $\tau$. Then for integers $m_2 \gg m_1 \gg m_0 \gg 0$, the following holds: let $\E$ be a $Y$-flat coherent sheaf on $X$ and let $y \in Y$ be a point. Then:
	\begin{enumerate}
		\item The sheaf $\E_y = \E \times_Y \spec \C$ is multi-Gieseker (semi)stable with respect to $\sigma$ if and only if $\E_y$ is pure, $(m_0, \underline{\LL}_y)$-regular and if $\mathrm{Hom}(T_y,\E_y)$ is a $\sigma$-(semi)stable $A$-module.
		\item Suppose in addition that $\sigma$ is positive, and suppose $\E_y$ is $\sigma$-semistable. Then
		$$ \mathrm{Hom}(T_y, \mathrm{gr}_{\sigma}(\E_y)) \cong \mathrm{gr}_{\sigma}(\mathrm{Hom}(T_y,\E_y)), $$
		where $\mathrm{gr}_{\sigma}$ denotes the associated graded object arising from a Jordan-H\"older filtration of $\E_y$ or of $\mathrm{Hom}(T_y,\E_y)$ respectively. In particular, $\mathrm{Hom}(T_y,-)$ preserves S-equivalence with respect to $\sigma$.
	\end{enumerate}
\end{theorem}

\subsection{The GIT Construction - Positive Case} \label{section: GRT construction}

We continue to fix the projective morphism $\pi : X \to Y$, the $\pi$-very ample invertible sheaves $\LL_1, \dots, \LL_k$ and the topological type $\tau$ of $Y$-flat coherent sheaves on $X$. Choose natural numbers $m_0, m_1$ and $m_2$ such that Proposition \ref{prop: embedding 1} and Theorem \ref{thm: comparison theorem} apply. Set the dimension vector to be
$$ \underline{d} = (P_1(m_1), P_1(m_2), \dots, P_k(m_1), P_k(m_2)). $$

We introduce the following notation: if $\E$ and $\F$ are coherent sheaves on $Y$ with $\F$ locally free, let $\underline{\mathrm{Hom}}_Y(\E,\F)$ denote the (linear) scheme representing the functor assigning to a $Y$-scheme $g : S \to Y$ the set $\mathrm{Hom}_{S}(g^{\ast}\E, g^{\ast}\F)$; such a scheme exists and is affine over $Y$ by \cite[Tag 08JY]{stacks-project}. With $H_{ij} = \pi_{\ast}(\LL_i^{-m_1} \otimes \LL_j^{m_2})$, form the $Y$-scheme
$$ R = \prod_{i,j=1}^k \underline{\mathrm{Hom}}_Y(\OO_Y^{d_{i1}} \otimes H_{ij}, \OO_Y^{d_{j2}}). $$
Let $p : R \to Y$ denote the structure morphism. The scheme $R$ carries a tautological family $\mathbb{M}$ of $A$-modules. By Proposition \ref{prop: embedding 2}, there exists a locally closed subscheme
$$ \iota_0 : R_{\tau}^{[reg]} \hookrightarrow R $$
parametrising the modules which appear in the image of the $\Hom(T,-)$-functor on the full subcategory of all $(m_1, \underline{\LL})$-regular coherent sheaves of topological type $\tau$. Then $\iota_0^{\ast} \mathbb{M} \otimes_A T$ is a $R_{\tau}^{[reg]}$-flat family of $(m_1, \underline{\LL})$-regular sheaves of topological type $\tau$. Let
$$ \iota : Q \hookrightarrow R_{\tau}^{[reg]} $$
denote the open subscheme parametrising those sheaves which are also $(m_0, \underline{\LL})$-regular, and let $\mathbb{F} = \iota^{\ast} \mathbb{M} \otimes_A T$. Let
$$ Q^{[\sigma-ss]} \subset Q $$
denote the open subscheme where the fibres of $\mathbb{F}$ are $\sigma$-semistable. Under the natural fibrewise action of $G \times Y$ on $R/Y$, where $G = \prod_{j=1}^k (GL_{d_{j1}}(\C) \times GL_{d_{j2}}(\C)) / \Delta$, the subschemes $Q$ and $Q^{[\sigma-ss]}$ are invariant. Endow $R$ with the linearisation $\OO_R(\chi_{\theta})$, where $\theta = c\theta_{\sigma}$ for some positive integer $c$ chosen so that $\theta$ is integral, and let $R^{\sigma-ss}$ be the corresponding relative GIT semistable locus (cf. Definition \ref{defn: relative semistable locus}). Let $\overline{Q^{[\sigma-ss]}}$ be the closure of $Q^{[\sigma-ss]}$ in $R$, and set
$$ (\overline{Q^{[\sigma-ss]}})^{\sigma-ss} := \overline{Q^{[\sigma-ss]}} \cap R^{\sigma-ss}. $$

\begin{theorem}[\cite{alvarezconsulking} Section 6, \cite{grebrosstoma} Section 9] \label{thm: relative moduli spaces}
	Suppose $\sigma$ is a positive stability parameter which is bounded with respect to $\tau$. Then there is an equality of schemes $(\overline{Q^{[\sigma-ss]}})^{\sigma-ss} := \overline{Q^{[\sigma-ss]}} \cap R^{\sigma-ss} = Q^{[\sigma-ss]}$, and there exists a good quotient $q : Q^{[\sigma-ss]} \to M_{\sigma, \tau}$ of $Y$-schemes for the action of $G$ on $Q^{[\sigma-ss]}$, where
	$$M_{\sigma, \tau} = \overline{Q^{[\sigma-ss]}} \sslash_{\theta_{\sigma}} G = Q^{[\sigma-ss]} \sslash G = \mathbf{Proj}_Y \left( \bigoplus_{m=0}^{\infty} (p_{\ast} (\OO_R(\chi_{\theta_{\sigma}}^m)|_{ \overline{Q^{[\sigma-ss]}}}))^G \right). $$
	The scheme $M_{\sigma, \tau}$ is projective over $Y$. Moreover, the closed points of $M_{\sigma, \tau}$ are in 1-1 correspondence with S-equivalence classes of $\sigma$-semistable sheaves over geometric fibres of $\pi : X \to Y$.
\end{theorem}

\begin{proof}
	Suppose first $Y = \spec \C$ is a point (that is, we are working with a fixed fibre over some geometric point $y \in Y$). That there is an equality $\overline{Q^{[\sigma-ss]}} \cap R^{\sigma-ss} = Q^{[\sigma-ss]}$, and that there is a good quotient $M_{\sigma, \tau}$ of $Q^{[\sigma-ss]}$ given by
	$$ M_{\sigma, \tau} = \mathrm{Proj} \left( \bigoplus_{m=0}^{\infty} H^0(\overline{Q^{[\sigma-ss]}}, \OO_R(\chi_{\theta_{\sigma}}^m)|_{ \overline{Q^{[\sigma-ss]}}})^G \right), $$
	both follow from \cite[Theorem 10.1]{grebrosstoma}. The projectivity (over $\spec \C$) of $M_{\sigma, \tau}$ is Theorem 9.6 of \emph{loc. cit.} (which in turn is proved in the same way as \cite[Proposition 6.6]{alvarezconsulking}, via establishing that the valuative criterion for properness holds), and the identification of the closed points of $M_{\sigma, \tau}$ with S-equivalence classes of semistable sheaves is \cite[Theorem 9.4]{grebrosstoma}. This yields the result in the fibrewise setting.
	
	In order to pass from the fibrewise setting to working over the quasi-projective scheme $Y$, we apply Lemma \ref{lem: technical git over base result}. Since the points of the relative semistable locus $R^{\sigma-ss} = R^{G-ss}(\chi_{\theta}/Y)$ are determined by the fibrewise semistable loci, we have an equality $\overline{Q^{[\sigma-ss]}} \cap R^{\sigma-ss} = Q^{[\sigma-ss]}$. We also have the existence of a good quotient $M_{\sigma, \tau}$ of $Q^{[\sigma-ss]}$, given by the relative projective spectrum
	$$ Q^{[\sigma-ss]} \sslash G = \mathbf{Proj}_Y \left( \bigoplus_{m=0}^{\infty} (p_{\ast} (\OO_R(\chi_{\theta_{\sigma}}^m)|_{ \overline{Q^{[\sigma-ss]}}}))^G \right). $$ 
	The description of the closed points of $M_{\sigma, \tau}$ is immediate from the fibrewise case. That $M_{\sigma, \tau}$ is proper, and thus projective, over $Y$ follows from the same valuative criterion argument used in the proof of \cite[Proposition 6.6]{alvarezconsulking}, which carries over to the relative case, exactly as in Section 6.5 of \emph{loc. cit.}.
\end{proof}	

\begin{remark}
	For a description of the moduli functor that the scheme $M_{\sigma, \tau}$ corepresents (at least in the case $Y = \spec \C$), see \cite[Section 9]{grebrosstoma} and \cite[Section 6.5]{alvarezconsulking}. The fibre of $M_{\sigma, \tau}$ over a point $y \in Y$ is exactly the moduli space of $\sigma$-semistable coherent sheaves on the fibre $X_y$, as constructed in \cite{grebrosstoma}.
\end{remark}

\subsection{The GIT Construction - Degenerate Case} \label{section: comparison of semistability degenerate} \label{section: GRT construction degenerate}

Let $\pi : X \to Y$, $\LL_1, \dots, \LL_k$ and $\tau$ be as in Section \ref{section: GRT construction}, and fix a bounded stability parameter $\sigma = (\underline{\LL}, \sigma_1, \dots, \sigma_k)$. Suppose instead that $\sigma$ is degenerate, that is some $\sigma_i = 0$. In this case, the second conclusion of Theorem \ref{thm: comparison theorem} no-longer applies. Instead we proceed as follows.

For integers $m_2 \gg m_1 \gg m_0 \gg 0$, the result of Proposition \ref{prop: embedding 1} and the first conclusion of Theorem \ref{thm: comparison theorem} still applies to $Y$-flat coherent sheaves on $X$ of topological type $\tau$. Fix such integers $m_i$. Relabelling indices if necessary, we may assume there exists a positive integer $k' < k$ with the property that $\sigma_j > 0$ for all $j \leq k'$ and $\sigma_j = 0$ for all $j > k'$. Let $\underline{\LL}' = (\LL_1, \dots, \LL_{k'})$ and $\sigma' = (\underline{\LL}', \sigma_1, \dots, \sigma_{k'})$. Then $\sigma'$ is a positive stability parameter, though with respect to a proper subset of the line bundles $\LL_1, \dots, \LL_k$. 

Form a subquiver $\mathcal{Q}'$ of the original quiver $\mathcal{Q}$ by taking the full subquiver with vertices $\mathcal{Q}_0' = \{v_i, w_j : i, j = 1, \dots, k'\}$. Let $T', A', \underline{d}', R', G'$ and $Q'$ be the objects associated with the quiver $\mathcal{Q}'$ (with $R'$ and $G'$ considered as schemes over $Y$), so that if $E$ is a $(m_1, \underline{\LL}')$-regular sheaf on $X$ of topological type $\tau$ then $\Hom(T', E)$ is a flat family of $A'$-modules of dimension vector $\underline{d}'$. After possibly increasing the $m_i$, we may also assume that Proposition \ref{prop: embedding 1} and Theorem \ref{thm: comparison theorem} apply to the positive stability parameter $\sigma'$.

The inclusion $\mathcal{Q}' \subset \mathcal{Q}$ gives rise to a projection
$$ \phi' : R \to R'. $$
Continue to set
$$ \theta_{j1} = \frac{\sigma_j}{\sum_i \sigma_i d_{i1}}, \quad \theta_{j2} = - \frac{\sigma_j}{\sum_i \sigma_i d_{i2}} $$
for all $j = 1, \dots, k$. The group $G$ acts on $R$ as $G' \times G''$, where $G''$ corresponds to the rows $j = k'+1, \dots, k$ of $\mathcal{Q}$. Letting $G''$ act on $R'$ trivially, the projection $\phi' : R \to R'$ is $G$-equivariant. Consider the subschemes
$$ D = \overline{Q^{[\sigma-ss]}} \subset R, \quad D^{\sigma-ss} = D \cap R^{\sigma-ss}, $$
where $Q^{[\sigma-ss]} \subset R_{\tau}^{[reg]}$ is as defined in Section \ref{section: GRT construction}, as well as the subschemes
$$ D' = \overline{(Q')^{[\sigma'-ss]}} \subset R', \quad (D')^{\sigma'-ss} = D' \cap (R')^{\sigma'-ss}. $$

\begin{prop}[\cite{grebrosstomamaster} Proposition 2.3] \label{prop: degenerate loci comparison}
	The projection $\phi' : R \to R'$ induces a $G'$-equivariant, $G''$-invariant map $D^{\sigma-ss} \to (D')^{\sigma'-ss}$, and the induced map $\hat{\phi}' : D^{\sigma-ss} \sslash G'' \to (D')^{\sigma'-ss} = (Q')^{[\sigma'-ss]}$ is a $G'$-equivariant isomorphism.
\end{prop}

\begin{proof}
    In the case $Y = \spec \C$, this is \cite[Proposition 2.3]{grebrosstomamaster}. The result in the relative setting then follows from the fibrewise setting by invoking Lemma \ref{lem: technical git over base result}, as in the proof of Theorem \ref{thm: relative moduli spaces}.
\end{proof}

\begin{cor}[\cite{grebrosstomamaster} Corollary 2.4] \label{cor: GRT degenerate setting}
	Suppose $\sigma$ is a degenerate bounded stability parameter. Then there exists a good quotient $M_{\sigma, \tau} := D^{\sigma-ss} \sslash G \cong (D')^{\sigma'-ss} \sslash G'$ of $Y$-schemes for the action of $G$ on $D^{\sigma-ss}$. 	The scheme $M_{\sigma, \tau}$ is projective over $Y$. Moreover, the closed points of $M_{\sigma, \tau}$ are in 1-1 correspondence with S-equivalence classes of $\sigma$-semistable sheaves over geometric fibres of $\pi : X \to Y$. \qed
\end{cor}

\subsection{The Master Space Construction} \label{section: master space}

Instead of considering a single stability parameter, suppose we have a finite set of bounded stability parameters $\mathfrak{S}$, with each $\sigma \in \mathfrak{S}$ defined with respect to the line bundles $\LL_i$. These stability parameters are allowed to be degenerate.

Fix positive integers $m_2 \gg m_1 \gg m_0 \gg 0$ such that for each stability parameter $\sigma \in \mathfrak{S}$, as well as for any positive stability parameter $\sigma'$ obtained by truncating a degenerate stability condition $\sigma \in \mathfrak{S}$, the conclusions of Proposition \ref{prop: embedding 1} and Theorem \ref{thm: comparison theorem} all hold; suppose further that $m_0, m_1$ and $m_2$ are chosen so that the formation of the locally free sheaves $H_{ij} = \pi_{\ast}(\LL_i^{-m_1} \otimes \LL_j^{m_2})$ is compatible with base change. Form the affine $Y$-scheme
$$ Z = Z_{\mathfrak{S}} = \bigcup_{\sigma \in \mathfrak{S}} \overline{Q^{[\sigma-ss]}}. $$
For each $\sigma \in \mathfrak{S}$, let
$$ Z^{\sigma-ss} := Z \cap R^{\sigma-ss}. $$
The moduli spaces $M_{\sigma, \tau}$, for $\sigma \in \mathfrak{S}$, are all GIT quotients of the master space $Z$.

\begin{theorem}[\cite{grebrosstoma} Theorem 10.1, \cite{grebrosstomamaster} Theorem 4.2] \label{thm: master space}
	For each $\sigma \in \mathfrak{S}$, there is an equality $Z^{\sigma-ss} = (\overline{Q^{[\sigma-ss]}})^{\sigma-ss} = \overline{Q^{[\sigma-ss]}} \cap R^{\sigma-ss}$. In particular, the moduli space $M_{\sigma, \tau}$ is isomorphic to the good quotient $Z^{\sigma-ss} \sslash G$ of $Y$-schemes.
\end{theorem}

\begin{proof}
    In the case where $Y$ is a point, this is \cite[Theorem 10.1]{grebrosstoma} (when $\sigma$ is positive) and \cite[Theorem 4.2]{grebrosstomamaster} (when $\sigma$ is degenerate). The result in the relative setting then follows as a consequence of Lemma \ref{lem: technical git over base result}, as in the proof of Theorem \ref{thm: relative moduli spaces}.
\end{proof}

\subsection{Existence of Further Quotients} \label{section: further quotients}

Assume now that $X$ and $Y$ both admit actions of a reductive group $H$, with $\pi$ an $H$-equivariant morphism. As in the hypotheses of Proposition \ref{prop: relative GIT result}, assume there exists an equivariant open immersion $\iota : Y \hookrightarrow \overline{Y}$ into a projective scheme $\overline{Y}$ acted on by $H$, with an ample linearisation $L$ such that (over $\spec \C$)
$$ \overline{Y}^{ss}(L) = \overline{Y}^s(L) = Y. $$ 
Assume in addition that the $\pi$-very ample line bundles $\LL_1, \dots, \LL_k$ are $H$-linearised. Fix a topological type $\tau$, and fix a finite set of bounded stability parameters $\mathfrak{S}$ defined with respect to the $\LL_i$.

As in Section \ref{section: GRT construction}, form the $Y$-scheme
$$ R = \prod_{i,j=1}^k \underline{\mathrm{Hom}}_Y(\OO_Y^{d_{i1}} \otimes H_{ij}, \OO_Y^{d_{j2}}), \quad H_{ij} = \pi_{\ast}(\LL_i^{-m_1} \otimes \LL_j^{m_2}), $$
with the natural numbers $m_0, m_1$ and $m_2$ chosen such that the formation of the locally free sheaves $H_{ij} = \pi_{\ast}(\LL_i^{-m_1} \otimes \LL_j^{m_2})$ is compatible with base change and such that Theorem \ref{thm: master space} applies for the fibrewise action of $G \times Y$ on $R/Y$, where as before $G = \prod_{j=1}^k (GL_{d_{j1}}(\C) \times GL_{d_{j2}}(\C)) / \Delta$. For each $\sigma \in \mathfrak{S}$, fix a positive integer $c = c_{\sigma}$ such that the vector $c \theta_{\sigma}$ is integral, and let $\chi_{\sigma} = \chi_{c \theta_{\sigma}}$ denote the corresponding character of $G$. The actions of $H$ on $X$, $Y$ and the $\mathcal{L}_i$ give rise to an induced action of $H$ on the sheaves $H_{ij}$ lifting the action of $H$ on $Y$. This in turn gives rise to an $H$-action on $R$, with the natural projection $p : R \to Y$ being $H$-equivariant. As $G$ acts fibrewise with respect to $p$, the actions of $G$ and $H$ on $R$ commute. In particular, we may consider each $\OO_R(\chi_{\sigma})$ as a $(G \times H)$-linearisation on $R$.

Let $Z = Z_{\mathfrak{S}}$ also be as above, and let $p_{\sigma} : Z^{\sigma-ss} \to M_{\sigma, \tau}$ be the good quotient over $Y$ given by Theorem \ref{thm: master space}; from Lemma \ref{lem: technical git over base result} this is explicitly given as
$$ M_{\sigma, \tau} = Z^{\sigma-ss} \sslash G = \mathbf{Proj}_Y \left( \bigoplus_{m=0}^{\infty} (p_{\ast} (\OO_R(\chi_{\sigma}^m)|_Z))^G \right). $$
Since each $\OO_R(\chi_{\sigma}^k)$ is naturally $H$-linearised and as the $G$ and $H$-actions on $R$ commute, $M_{\sigma, \tau}$ admits a residual $H$-action, and the relatively ample line bundle $\OO_M(1)$ on $M_{\sigma, \tau}$ arising from the $\mathbf{Proj}_Y$-construction admits an induced $H$-linearisation. If $q_{\sigma} : M_{\sigma, \tau} \to Y$ denotes the ($H$-equivariant) structure morphism, Proposition \ref{prop: relative GIT result} is applicable to $q_{\sigma}$; in particular there exists a geometric quotient $M_{\sigma, \tau} \sslash H$ of $M_{\sigma, \tau}$ and a projective morphism $\hat{q}_{\sigma} : M_{\sigma, \tau} \sslash H \to Y \sslash H$ induced by $q_{\sigma}$. The composition $Z^{\sigma-ss} \to M_{\sigma, \tau} \to M_{\sigma, \tau} \sslash H$ is then a good quotient for the $(G \times H)$-action on $Z^{\sigma-ss}$. This proves the following result.

\begin{prop} \label{prop: further quotients}
	Let $H$ be a reductive group. Let $\pi : X \to Y$ be an equivariant projective morphism between quasi-projective schemes acted on by $H$. Assume there exists an equivariant open immersion $\iota : Y \hookrightarrow \overline{Y}$ into a projective scheme $\overline{Y}$ acted on by $H$, with an ample linearisation $L$ such that $\overline{Y}^{ss}(L) = \overline{Y}^s(L) = Y$. Fix $H$-linearised, $\pi$-very ample invertible sheaves $\LL_1, \dots, \LL_k$, fix a topological type $\tau$ of flat sheaves on the fibres of $\pi$ and fix a finite set of bounded stability parameters $\mathfrak{S}$, with both $\tau$ and $\mathfrak{S}$ defined with respect to the $\LL_i$. Form the subscheme $Z = Z_{\mathfrak{S}}$ as described in Section \ref{section: master space}. Finally, fix $\sigma \in \mathfrak{S}$.
	
	There then exists a projective good quotient $Z^{\sigma-ss} \sslash G \times H = M_{\sigma, \tau} \sslash H$ of $Z^{\sigma-ss}$, and there exists a natural projective morphism $\hat{q}_{\sigma} : Z^{\sigma-ss} \sslash G \times H \to Y \sslash H = \overline{Y}^{s}(L) / H$. If $y \in Y$ is a closed point, the fibre of $\hat{q}_{\sigma}$ over $y$ is isomorphic to $(M_{\sigma, \tau})_y / H_y$. \qed
\end{prop}

The good quotient $Z^{\sigma-ss} \sslash G \times H$ can be understood as the quotient of a relative GIT semistable locus for a $G \times H$-linearisation on $Z$. 

\begin{prop} \label{prop: quotient of Z^{sigma-ss} is a quotient of a relative ss locus}
	There exist positive integers $a, N > 0$ such that for all $\sigma \in \mathfrak{S}$, there is an equality
	$$ Z^{\sigma-ss} = Z^{G \times H-ss}(p^{\ast}(L^{aN}|_Y) \otimes \OO_Z(\chi_{\sigma}^N)/Y) $$
	of open subschemes of $Z$.
\end{prop}
	
\begin{proof}
	Fix $\sigma \in \mathfrak{S}$. From the proof of Proposition \ref{prop: relative GIT result}, for all $a \gg 0$, each point of $M_{\sigma, \tau}$ is $H$-stable with respect to the linearisation $q_{\sigma}^{\ast}(L^a|_Y) \otimes \OO_M(1)$. By Proposition \ref{prop: alper and git} some power $\OO_M(N_{\sigma})$ of $\OO_M(1)$ pulls back along the good quotient $p_{\sigma}$ to $\OO_Z(\chi_{\sigma}^{N_{\sigma}})|_{Z^{\sigma-ss}}$; we may assume without loss of generality that $N_{\sigma} = N$ is the same for each $\sigma \in \mathfrak{S}$.
	
	Fixing $a > 0$ sufficiently large, we have equalities
	$$ Z^{\sigma-ss} = Z^{G-ss}(\OO_Z(\chi_{\sigma}^N)/Y) = Z^{G-ss}(p^{\ast}(L^{aN}|_Y) \otimes \OO_Z(\chi_{\sigma}^N)/Y); $$
	the first equality follows from $Z^{\sigma-ss} = Z \cap R^{\sigma-ss} = Z \cap R^{G-ss}(\OO_R(\chi_{\sigma}^N)/Y)$ and the second equality holds since the relative $G$-semistable locus is determined fibrewise (cf. Lemma \ref{lem: technical git over base result}). From this and from the observation that $G \times H$-semistability implies $G$-semistability, it suffices to show that there is an inclusion
	$$ Z^{G-ss}(p^{\ast}(L^{aN}|_Y) \otimes \OO_Z(\chi_{\sigma}^N)/Y) \subset Z^{G \times H-ss}(p^{\ast}(L^{aN}|_Y) \otimes \OO_Z(\chi_{\sigma}^N)/Y) $$
	of open subschemes of $Z$. 
	
	If $z \in Z^{G-ss}(p^{\ast}(L^{aN}|_Y) \otimes \OO_Z(\chi_{\sigma}^N)/Y)$ lies over $\overline{z} \in M_{\sigma, \tau}$, as $\overline{z}$ is $H$-stable there exists $k > 0$ and a section $f \in H^0(M_{\sigma, \tau}, q_{\sigma}^{\ast}(L^{akN}|_Y) \otimes \OO_M(kN))^H$ such that the non-vanishing locus $(M_{\sigma, \tau})_f$ is affine and $\overline{z} \in (M_{\sigma, \tau})_f$. If $f' := p_{\sigma}^{\ast}(f)$, then $f'$ is a $G \times H$-invariant section of $p^{\ast}(L^{aNk}|_Y) \otimes \OO_Z(\chi_{\sigma}^{Nk})$ with affine non-vanishing locus $p_{\sigma}^{-1}((M_{\sigma, \tau})_f)$ containing $z$. Hence $z \in Z^{G \times H-ss}(p^{\ast}(L^{aN}|_Y) \otimes \OO_Z(\chi_{\sigma}^N)/Y)$.
\end{proof}

\section{The GIT Construction of the Moduli Spaces} \label{section: carrying out the construction}

With the setup of the previous section, we are ready to prove Theorem \ref{thm: Theorem A}.

\subsection{Setup} \label{section: GIT construction setup} Fix a projective variety $X$ and consider the stack of stable maps $\overline{\M}_{g,n}(X, \beta)$. Let $\pi_{\UU} : \UU \overline{\M}_{g,n}(X, \beta) \to \overline{\M}_{g,n}(X, \beta)$ denote the universal family over $\overline{\M}_{g,n}(X, \beta)$. Fix $\pi_{\UU}$-ample $\Q$-invertible sheaves $\LL_1, \dots, \LL_k$. For each $\sigma \in \Sigma := (\Q^{\geq 0})^k \setminus \{0\}$ let $\LL_{\sigma} = \bigotimes_i \LL_i^{\sigma_i}$, and consider the stacks
$$ \overline{\J}ac(\sigma) := \overline{\J}ac_{g,n,d,r}^{ss}(X, \beta)(\LL_{\sigma}) \text{   and   } \overline{\J}(\sigma) := \overline{\J}_{g,n,d,r}^{ss}(X, \beta)(\LL_{\sigma}). $$
Since replacing each $\LL_i$ with $\LL_i^m$ (where $m$ is a positive integer) does not alter the stability condition, we assume without loss of generality that each $\LL_i$ is a genuine invertible sheaf and is relatively very ample.

\begin{lemma} \label{lem: boundedness result}
	The collection of all degree $d$, rank $r$, torsion-free coherent sheaves $E$ over objects of the stack of stable maps $\overline{\M}_{g,n}(X,\beta)$ for which there exists some $\sigma \in \Sigma$ such that $E$ is $\sigma$-semistable, forms a bounded family.
\end{lemma}

\begin{proof}
	Without loss of generality it is enough to consider stability conditions in the compact convex set $\Sigma' := \left\{ \sigma \in \Sigma : \sum_{i=1}^k \sigma_i = 1 \right\}$. Since the stable maps in $\overline{\mathcal{M}}_{g,n}(X, \beta)$ vary in a bounded schematic family, there are only finitely many possible topological types of nodal curves appearing in such a stable map. As such, there exists a constant $\mu_0$ such that if $E$ is a degree $d$, uniform rank $r$, torsion-free coherent sheaf lying over a stable map in $\overline{\mathcal{M}}_{g,n}(X, \beta)$ (with underlying nodal curve $C$) which is $\LL_1$-semistable, then 
	$$ \mu^{\LL_i}(E) = \frac{\chi(E)}{r \sum_{j=1}^{\rho} \deg_{C_j} L_1} \leq \mu_0; $$
	here $L_i$ denotes the pullback of $\LL_i$ to $C$, and $C_1, \dots, C_{\rho}$ are the irreducible components of $C$.
	
	Take $\sigma \in \Sigma'$, and suppose $E$ is a degree $d$, uniform rank $r$, torsion-free coherent sheaf lying over a stable map in $\overline{\mathcal{M}}_{g,n}(X, \beta)$ which is $\sigma$-semistable. If $E$ is semistable with respect to $\LL_1$ then $\mu(E) \leq \mu_0$. Otherwise, let $E' \subset E$ denote the maximally destabilising subsheaf of $E$ with respect to $\LL_1$ (as $E$ is torsion-free, then $E'$ is of pure dimension $1$). For any $t \in [0,1]$, set
	$$ \sigma(t) := (1 - t + t \sigma_1, t \sigma_2, \dots, t \sigma_k) \in \Sigma'. $$
	Since $\mu^{\LL_1}(E') > \mu^{\LL_1}(E)$ and $\mu^{\sigma}(E') \leq \mu^{\sigma}(E)$, by continuity there must exist $t_0 \in [0,1]$ with $\mu^{\sigma(t_0)}(E') = \mu^{\sigma(t_0)}(E)$. This implies
	\begin{equation} \label{eqn: L_1 slope of MDSS}
		\mu^{\LL_1}(E') = \frac{\chi(E')}{\sum_j r_j(E') \deg_{C_j} L_1} = \frac{\sum_{i,j} \sigma(t_0)_i r_j(E') \deg_{C_j} L_i}{r (\sum_j r_j(E') \deg_{C_j} L_1) (\sum_{i,j} \sigma(t_0)_i \deg_{C_j} L_i)} \chi(E).
	\end{equation}
	Since each $0 \leq r_j(E') \leq r$ and are not all zero, since $\chi(E)$ is a constant, since the curve $C$ varies in a bounded family and since the set of stability conditions $\Sigma'$ is compact, it follows from Equation \eqref{eqn: L_1 slope of MDSS} that there exists a constant $\mu_1 \geq \mu_0$, depending only on $d, r, g, n, X, \beta$ and $\underline{\LL}$, such that the following holds: for any degree $d$, uniform rank $r$, torsion-free coherent sheaf $E$ lying over a stable map in $\overline{\mathcal{M}}_{g,n}(X, \beta)$ which is $\sigma$-semistable for some $\sigma \in \Sigma'$,
	\begin{enumerate}
		\item either $E$ is $\LL_1$-semistable, and hence $\mu^{\LL_1}(E) \leq \mu_1$; or
		\item $E$ is $\LL_1$-unstable, and the $\LL_1$-maximally destabilising subsheaf $E' \subset E$ satisfies $\mu^{\LL_1}(E') \leq \mu_1$.
	\end{enumerate}
It follows by \cite[Theorem 3.3.7]{huybrechts_lehn} that the collection of all such sheaves $E$ is bounded, which proves the result.
\end{proof}

As in Section \ref{section: BS construction}, fix an embedding $X \subset \PP^b$, set $\OO_X(1) = \OO_{\PP^b}(1) |_X$ and consider the universal family $\phi_{\UU} : \UU I \to I$ over the quasi-projective scheme $I$. Pulling back the sheaves $\LL_i$ along the forgetful morphism $I \to \overline{\M}_{g,n}(X, \beta)$ gives $\phi_{\UU}$-very ample invertible sheaves $\LL_i'$ on $\UU I$, which are easily seen to be equivariant with respect to the $GL(W)$-action on $I$. Take $\tau$ to be the topological type of degree $d$, rank $r$ torsion-free sheaves over the fibres of $\phi_{\UU}$ with respect to the $\LL_i'$ (cf. Section \ref{section: sheaves on nodal curves}). By Theorem \ref{thm: baldwin swinarski}, there exists an equivariant open immersion $I \hookrightarrow \overline{I}$ and a linearisation $L = L_{m_W, m_b, m_{\mathrm{pts}}}$ on $\overline{I}$ for the induced $SL(W)$-action for which
$$ \overline{I}^{ss}(L) = \overline{I}^s(L) = I, $$
with $\overline{M}_{g,n}(X, \beta) \cong I \sslash SL(W)$ isomorphic to the resulting geometric GIT quotient. In particular, we will be able to apply the result of Proposition \ref{prop: further quotients} to the universal family $\phi_{\UU} : \UU I \to I$ and the linearised sheaves $\LL_i'$.

Fix a finite subset $\mathfrak{S} \subset \Sigma$. We now restrict attention to stability parameters $\sigma \in \mathfrak{S}$. As the collection of sheaves in question is bounded (cf. Lemma \ref{lem: boundedness result}), we may pick natural numbers $m_0$, $m_1$ and $m_2$ such that Theorem \ref{thm: master space} applies for the projective morphism $\phi_{\UU} : \UU I \to I$ and the sheaves $\LL_i'$; from our choice of $m_i$ the formation of the locally free sheaves $H_{ij} = (\phi_{\UU})_{\ast}((\LL_i')^{-m_1} \otimes (\LL_j')^{m_2})$ is compatible with base change, and any sheaf appearing in $\overline{\J}_{g,n,d,r}(X, \beta)$ which is semistable with respect to any $\sigma \in \mathfrak{S}$ is $(m_0, \underline{\LL'})$-regular.

Form the $I$-schemes
$$ Z = Z_{\mathfrak{S}} = \bigcup_{\sigma \in \mathfrak{S}} \overline{Q^{[\sigma-ss]}} \subset R = \prod_{i,j=1}^k \underline{\mathrm{Hom}}_I(\OO_I^{d_{i1}} \otimes H_{ij}, \OO_I^{d_{j2}}), $$
with the dimension vector given by
$$ \underline{d} = (d_{11}, d_{12}, \dots, d_{k1}, d_{k2}) = (P_1(m_1), P_1(m_2), \dots, P_k(m_1), P_k(m_2)). $$
We endow $R/I$ with the usual fibrewise actions of the groups $G \times I$ and $\hat{G} \times I$, where
$$ \hat{G} = \prod_{j=1}^k (GL_{d_{j1}}(\C) \times GL_{d_{j2}}(\C)), \quad G = \hat{G} / \Delta. $$
As in Section \ref{section: further quotients}, the $\hat{G}$-action on $R$ extends to an action of $\hat{G} \times GL(W)$. The $\hat{G} \times GL(W)$-action naturally descends to an action of $G \times PGL(W)$, since the diagonal one-parameter subgroups of both $\hat{G}$ and $GL(W)$ act trivially.

We now restrict attention further to sheaves of uniform rank $r$.

\begin{lemma} \label{lem: flat families uniform rank}
	Let $\mathcal{C} \to S$ be a flat, projective family of genus $g$ prestable curves. Let $\mathcal{F}$ be an $S$-flat torsion free coherent sheaf on $\mathcal{C}$, each having Hilbert polynomial $P$ with respect to a fixed choice of an $S$-ample line bundle on $\mathcal{C}$. Then there exists an open and closed subscheme $S_r \subset S$ such that, if $g : T \to S$ is any morphism of schemes, then $g$ factors through $S_r \subset S$ if and only if for all geometric points $t \in T$, the sheaf $\mathcal{F}_{g(t)}$ is of uniform rank $r$ over the curve $\mathcal{C}_{g(t)}$.
\end{lemma}

\begin{proof}
	This is essentially the content of \cite[Lemma 8.1.1]{pand}.
\end{proof}

Fix a stability parameter $\sigma \in \mathfrak{S}$. If $\sigma$ is degenerate, let $\phi' = \phi_{\sigma}' : R \to R' = R_{\sigma}'$ be the projection map described in Section \ref{section: comparison of semistability degenerate}, and let $\sigma'$ be the corresponding positive stability parameter. By Theorem \ref{thm: relative moduli spaces} and Proposition \ref{prop: degenerate loci comparison} we have equalities
\begin{equation*}
	Z^{\sigma-ss} = \begin{cases} Q^{[\sigma-ss]} & \text{if } \sigma \text{ is positive,} \\ (\phi')^{-1}((Q')^{[\sigma'-ss]}) & \text{if } \sigma \text{ is degenerate}. \end{cases}
\end{equation*}

Both $Q^{[\sigma-ss]}$ and $(Q')^{[\sigma'-ss]}$ parametrise flat families of torsion-free sheaves on prestable curves, so by Lemma \ref{lem: flat families uniform rank} there are open and closed subschemes, respectively denoted $Q_r^{[\sigma-ss]}$ and $(Q_r')^{[\sigma'-ss]}$, parametrising all sheaves in the tautological family $\mathbb{F}$ which are of uniform rank $r$. Let
\begin{equation*}
	Z_r^{\sigma-ss} = \begin{cases} Q_r^{[\sigma-ss]} & \text{if } \sigma \text{ is positive,} \\ (\phi')^{-1}((Q_r')^{[\sigma'-ss]}) & \text{if } \sigma \text{ is degenerate}. \end{cases}
\end{equation*}
In both cases, the scheme $Z_r^{\sigma-ss}$ is open and closed in $Z^{\sigma-ss}$. 

\subsection{Stacks of $\sigma$-semistable Sheaves as Quotient Stacks} In the case where $\sigma$ is positive, $Z_r^{\sigma-ss} = Q_r^{[\sigma-ss]}$ gives a smooth presentation of both $\overline{\J}ac(\sigma)$ and $\overline{\J}(\sigma)$, exhibiting both of these as quotient stacks.

\begin{prop} \label{prop: quotient stack description 1}
	Suppose $\sigma \in \mathfrak{S}$ is positive. Then there is an isomorphism of stacks over $\overline{\M}_{g,n}(X, \beta)$
	$$ \overline{\J}ac(\sigma) := \overline{\J}ac_{g,n,d,r}^{ss}(X, \beta)(\LL_{\sigma}) \cong [Z_r^{\sigma-ss} / \hat{G} \times PGL(W)]. $$
\end{prop}

\begin{proof}
	Let $[Z_r^{\sigma-ss} / \hat{G} \times PGL(W)]^{\mathrm{pre}}$ denote the quotient pre-stack associated to the $(\hat{G} \times PGL(W))$-action on $Z_r^{\sigma-ss} = Q_r^{[\sigma-ss]}$, so that $[Z_r^{\sigma-ss} / \hat{G} \times PGL(W)]$ is the stack associated to this CFG. Given a morphism of schemes $S \to Z_r^{\sigma-ss}$ over $I$, the pullback of the tautological family over $Q_r^{[\sigma-ss]}$ gives rise to a flat family of rank $r$, degree $d$, $\sigma$-semistable torsion-free coherent sheaves, over the family of $n$-pointed genus $g$ stable maps to $X$ embedded in $\PP(W) \times \PP^r$, obtained by pulling back the universal family $\phi_{\UU} : \UU I \to I$ to $S$, giving an object of $\overline{\J}ac(\sigma)(S)$. This defines a morphism $Z_r^{\sigma-ss} \to \overline{\J}ac(\sigma)$ of stacks over $\overline{\M}_{g,n}(X, \beta)$ which is easily seen to be invariant with respect to the action of $\hat{G} \times PGL(W)$, giving a morphism of CFGs
	$$ \Phi : [Z_r^{\sigma-ss} / \hat{G} \times PGL(W)]^{\mathrm{pre}} \to \overline{\J}ac(\sigma). $$
	By the uniqueness of stackifications, it suffices to show that the stackification $\tilde{\Phi}$ of $\Phi$ is an equivalence of CFGs.
	
	\emph{$\tilde{\Phi}$ is fully faithful:} Since stackification is fully faithful, it suffices to show that $\Phi$ is fully faithful. Suppose we are given morphisms $\gamma_i : S \to Z_r^{\sigma-ss}$ ($i = 1, 2$) such that the resulting families of semistable sheaves over stable maps
	$$ (h_i : \mathcal{C}_i \to S; \sigma_1^i, \dots, \sigma_n^i; f_i : \mathcal{C}_i \to X; \F_i) \in \overline{\J}ac(\sigma)(S) $$
	are isomorphic. Fix such an isomorphism; that is, fix an isomorphism of $S$-schemes $g : \mathcal{C}_1 \to \mathcal{C}_2$ (which is compatible with the sections $\sigma_j^i$ and the morphisms to $X$) and fix an isomorphism of coherent sheaves $\alpha : \F_1 \to g^{\ast} \F_2$.
	
	The isomorphism $g$ induces an isomorphism of invertible sheaves $\mathcal{N}_1 \cong g^{\ast} \mathcal{N}_2$, where
	$$ \mathcal{N}_i = \omega_{\mathcal{C}_i/S}(\sigma_1^i(S) + \cdots + \sigma_n^i(S)) \otimes f_i^{\ast}(\OO_X(3)). $$
	As the restriction of $\LL_i^{10}$ to a geometric fibre is non-special, the sheaves $(h_i)_{\ast}(\LL_i^{10})$ are locally free of rank $e - g + 1 = \dim W$. By the universal property of $I$ there are line bundles $M_i \in \Pic S$ and isomorphisms $$ \OO_{\PP(W)}(1) \otimes \OO_{\PP^b}(1) \otimes \OO_{\mathcal{C}_i} \cong (\omega_{\mathcal{C}_i/S}(\sigma_1^i(S) + \cdots + \sigma_n^i(S)))^{10} \otimes f_i^{\ast}(\OO_X(31)) \otimes \OO_{\mathcal{C}_i} \otimes h_i^{\ast} M_i $$
	of sheaves on $\mathcal{C}_i$. Pull back these isomorphisms along the projection $p_W : \PP(W) \times \PP^b \to \PP(W)$ then push forward to $S$ to obtain isomorphisms
	$$ (h_i)_{\ast}(\OO_{\PP(W)}(1) \otimes \OO_{p_W(\mathcal{C}_i)}) \cong (h_i)_{\ast}(\mathcal{N}_i^{10}) \otimes M_i. $$
	As each point of $I$ corresponds (after composing with the projection $p_W$) to a non-degenerate curve in $\PP(W)$, the natural morphisms
	$$ H^0(\PP(W), \OO_{\PP(W)}(1)) \otimes \OO_S \to (h_i)_{\ast}(\OO_{\PP(W)}(1) \otimes \OO_{p_W(\mathcal{C}_i)}) $$
	are both isomorphisms. The sequence of isomorphisms
	\begin{align*}
		H^0(\PP(W), \OO_{\PP(W)}(1)) \otimes \OO_S &\cong (h_1)_{\ast}(\OO_{\PP(W)}(1) \otimes \OO_{p_W(\mathcal{C}_1)}) \\
		&\cong (h_1)_{\ast}(\mathcal{N}_1^{10}) \otimes M_1 \\
		&\cong (h_2)_{\ast}(\mathcal{N}_2^{10}) \otimes M_1 \\
		&\cong (h_2)_{\ast}(\OO_{\PP(W)}(1) \otimes \OO_{p_W(\mathcal{C}_2)}) \otimes M_1 \otimes M_2^{-1} \\
		&\cong H^0(\PP(W), \OO_{\PP(W)}(1)) \otimes M_1 \otimes M_2^{-1}
	\end{align*}
yields an automorphism of $S \times \PP(W)$ sending the curve $p_W(\mathcal{C}_1)$ to $p_W(\mathcal{C}_2)$, along with a corresponding $S$-valued point of $PGL(W)$, depending only on $g$, $\gamma_1$ and $\gamma_2$. 
	
	On the other hand, the isomorphism $\alpha : \F_1 \to g^{\ast} \F_2$ gives rise to an isomorphism of $A$-modules $\mathcal{H}om(T, \F_1) \cong \mathcal{H}om(T, g^{\ast} \F_2)$. By Proposition \ref{prop: embedding 2}, the $A$-module $\mathcal{H}om(T, \F_1)$ is canonically isomorphic to the pullback of the tautological family $\mathbb{M}|_{Z_r^{\sigma-ss}}$ along the morphism $\gamma_1$, and similarly for the $A$-module $\mathcal{H}om(T, g^{\ast} \F_2)$, so there is an induced isomorphism between the (free) $A$-modules $\gamma_1^{\ast}(\mathbb{M}|_{Z_r^{\sigma-ss}})$ and $\gamma_2^{\ast}(\mathbb{M}|_{Z_r^{\sigma-ss}})$. Such an isomorphism determines an $S$-valued point of $\hat{G}$, depending only on $\alpha$, $\gamma_1$ and $\gamma_2$. It follows that $\Phi$ is fully faithful.
	
	\emph{$\tilde{\Phi}$ is essentially surjective:} Take an object
	\begin{equation} \label{eqn: essential surjectivity}
		(h : \mathcal{C} \to S; \sigma_1, \dots, \sigma_n; f : \mathcal{C} \to X; \F) \in \overline{\J}ac(\sigma)(S).
	\end{equation}
	Let $\mathcal{N}$ be the invertible sheaf
	$$ \mathcal{N} = \omega_{\mathcal{C}/S}(\sigma_1(S) + \cdots + \sigma_n(S)) \otimes f^{\ast}(\OO_X(3)), $$
	and consider the locally free sheaf $h_{\ast} \mathcal{N}^{10}$. From our choice of natural numbers $m_i$ the sheaves $h_{\ast}(\F \otimes (\LL_i')_S^{m_1})$, $h_{\ast}(\F \otimes (\LL_j')_S^{m_2})$ ($i, j = 1, \dots, k$) appearing in the $A$-module $\mathcal{H}om(T,\F)$ are also all locally free $\OO_S$-modules. Choose an open cover $\{S_i\}$ of $S$ which simultaneously trivialises the sheaves
	\begin{equation} \label{eqn: list of sheaves}
		 h_{\ast} \mathcal{N}^{10}, \quad h_{\ast}(\F \otimes (\LL_i')_S^{m_1}), \ \ i = 1, \dots, k, \quad h_{\ast}(\F \otimes (\LL_j')_S^{m_2}), \ \ j = 1, \dots, k.
	\end{equation}
	As $\overline{\J}ac(\sigma)$ is a stack, it suffices to show that the restriction of the object \eqref{eqn: essential surjectivity} to each $S_i \subset S$ lies in the essential image of $Z_r^{\sigma-ss} \to \overline{\mathcal{J}}ac(\sigma)$. Without loss of generality, we may therefore assume that all of the sheaves in \eqref{eqn: list of sheaves} are free $\OO_S$-modules.
	
	Pick an isomorphism $W \otimes \OO_S \cong h_{\ast} \mathcal{N}^{10}$. Pulling back to $\mathcal{C}$, we obtain a quotient
	\begin{equation*}
		W \otimes \OO_{\mathcal{C}} \cong h^{\ast} h_{\ast} \LL^{10} \to \LL^{10} \to 0.
	\end{equation*}
	This quotient embeds $\mathcal{C}$ inside $S \times \PP(W)$ as a family of non-degenerate curves; in turn $\mathcal{C}$ embeds inside $S \times \PP(W) \times \PP^b$ via the graph of $f : \mathcal{C} \to X \subset \PP^b$. From the universal property of $I$ we obtain a morphism $\gamma' : S \to I$. On the other hand, by Theorem \ref{thm: comparison theorem} the free $A$-module $\mathcal{H}om(T,\F)$ is (fibrewise) $\sigma$-semistable, so from the universal property of $Z_r^{\sigma-ss} = Q_r^{[\sigma-ss]}$ we obtain a morphism $\gamma : S \to Z_r^{\sigma-ss}$ which lifts the morphism $\gamma' : S \to I$. Consequently $\tilde{\Phi}$ is essentially surjective.
\end{proof}

\begin{cor} \label{cor: quotient stack description 2}
	Suppose $\sigma \in \mathfrak{S}$ is positive. Then there is an isomorphism of stacks over $\overline{\M}_{g,n}(X, \beta)$
	$$ \overline{\J}(\sigma) := \overline{\J}_{g,n,d,r}^{ss}(X, \beta)(\LL_{\sigma}) \cong [Z_r^{\sigma-ss} / G \times PGL(W)]. $$
	In particular, each of the stacks $\overline{\J}(\sigma)$ is algebraic and of finite type over $\C$.
\end{cor}

\begin{proof}
	Under the isomorphism of Proposition \ref{prop: quotient stack description 1}, the diagonal one-parameter subgroup $\G_m \cong \Delta \subset \hat{G}$ corresponds to the standard central copy of $\G_m$ in the automorphism groups of the sheaves appearing in $\overline{\J}ac(\sigma)$. The desired isomorphism is then obtained by rigidifying the isomorphism of Proposition \ref{prop: quotient stack description 1} with respect to these copies of $\G_m$.
\end{proof}

\begin{remark}
	We only claim that $Z^{\sigma-ss}$ gives a smooth presentation when $\sigma$ is positive. However, in the degenerate case the scheme $(Q_r')^{[\sigma'-ss]}$ can be used to exhibit $\overline{\J}ac(\sigma)$ and $\overline{\J}(\sigma)$ as quotient stacks, after replacing $\hat{G}$ and $G$ with $\hat{G}'$ and $G'$ respectively. 
\end{remark}

\subsection{The Good Moduli Spaces are GIT Quotients}

With the notation as above, we now show that the stack $\overline{\J}(\sigma)$ admits a projective good moduli space, which is a good quotient of the parameter space $Z_r^{\sigma-ss}$.

\begin{theorem} \label{thm: good moduli spaces as GIT quotients}
	Let $\sigma \in \mathfrak{S}$ be either positive or degenerate. The stack $\overline{\J}(\sigma)$ admits a projective good moduli space
	$$ \overline{J}(\sigma) = \overline{J}_{g,n,d,r}^{ss}(X, \beta)(\LL_{\sigma}), $$
	which is a coarse moduli space if all $\sigma$-semistable sheaves appearing in $\overline{\J}(\sigma)$ are $\sigma$-stable. Moreover, $\overline{J}(\sigma)$ is a good quotient of the scheme $Z_r^{\sigma-ss}$ under the action of the group $G \times SL(W)$:
	$$ \overline{J}(\sigma) \cong Z_r \sslash_{\theta_{\sigma}} G \times SL(W) := Z_r^{\sigma-ss} \sslash G \times SL(W). $$
    This good quotient coincides with the quotient of a relative (over $I$) GIT semistable locus for the action of $G \times SL(W)$ on $Z_r$.
\end{theorem}

\begin{proof}
	First note that if all $\sigma$-semistable sheaves appearing in $\overline{\J}(\sigma)$ are $\sigma$-stable then $\overline{\J}(\sigma)$ is a Deligne-Mumford stack (cf. Section \ref{section: substacks defined by polarisations}), hence any good moduli space is automatically a coarse moduli space.
	
	Since $\overline{I}^{SL(W)-ss}(L) = \overline{I}^{SL(W-s)}(L) = I$, after first replacing $Z$ with the open and closed subscheme $Z_r$, the results of Propositions \ref{prop: further quotients} and \ref{prop: quotient of Z^{sigma-ss} is a quotient of a relative ss locus} are applicable to the morphism $Z_r \to I$. As such, there exists a $G \times SL(W)$-linearisation $\tilde{L}_{\sigma}$ on $Z_r$, an equality $Z_r^{\sigma-ss} = Z^{G \times SL(W)-ss}(\tilde{L}_{\sigma}/I)$ of semistable loci, and a projective good quotient $Z_r^{\sigma-ss} \sslash G \times SL(W)$ for the action of $G \times SL(W)$ on $Z_r^{\sigma-ss}$.
    
     Since the group $GL(W)$ acts on $Z_r^{\sigma-ss}$ through the quotient $PGL(W)$, there is a canonical identification of good quotients
    $$ Z_r^{\sigma-ss} \sslash G \times SL(W) \equiv Z_r^{\sigma-ss} \sslash G \times PGL(W). $$
    If $\sigma$ is positive, the isomorphism $\overline{J}(\sigma) \cong Z_r \sslash_{\theta_{\sigma}} G \times SL(W)$ follows by combining Corollary \ref{cor: quotient stack description 2} with Proposition \ref{prop: good moduli spaces and good quotients}. If instead $\sigma$ is degenerate, we apply Corollary \ref{cor: quotient stack description 2} with Proposition \ref{prop: good moduli spaces and good quotients} to $Z_r^{\sigma'-ss}$ in place of $Z_r^{\sigma-ss}$ and $G'$ in place of $G$, then apply Corollary \ref{cor: GRT degenerate setting} to obtain an isomorphism of good quotients
    $$ Z_r^{\sigma-ss} \sslash G \times SL(W) \cong Z_r^{\sigma'-ss} \sslash G' \times SL(W) \equiv Z_r^{\sigma'-ss} \sslash G' \times PGL(W). $$
    This completes the proof.
\end{proof}

\subsection{Closed Points of the Moduli Spaces} Continue to fix the stability condition $\sigma$. Let $\zeta$ be a stable map in $\overline{\M}_{g,n}(X, \beta)$, and let $[\zeta]$ be the corresponding point in the coarse moduli space $\overline{M}_{g,n}(X, \beta)$. Denote the pullback of $\overline{\J}(\sigma)$ along $\zeta : \spec \C \to \overline{\M}_{g,n}(X, \beta)$ as
$$ \overline{\J}_{\zeta}(\sigma) = \overline{\J}_{\zeta, d, r}^{ss}(\LL_{\sigma} |_{\zeta}). $$
Fix a point $j_0 \in I$ corresponding to an embedded stable map
$$ (C \subset \PP(W) \times \PP^b; x_1, \dots, x_n; \mathrm{pr}_{\PP^b} : C \to X \subset \PP^b) $$
whose underlying stable map is isomorphic to $\zeta$. If $(Z_r)_{j_0}^{\sigma-ss}$ denotes the pullback of the master space $Z_r^{\sigma-ss}$ along $j_0$, the good quotient $(Z_r)_{j_0}^{\sigma-ss} \sslash G$ is isomorphic to the good moduli space $\overline{J}_{\zeta}(\sigma)$ of $\overline{\J}_{\zeta}(\sigma)$, and is the moduli space of $\sigma$-semistable degree $d$, rank $r$ torsion-free sheaves over the nodal curve $C$ appearing in \cite{grebrosstoma} and \cite{grebrosstomamaster}. By Theorem \ref{thm: relative moduli spaces} and Corollary \ref{cor: GRT degenerate setting}, the closed points of $\overline{J}_{\zeta}(\sigma)$ are in 1-1 correspondence with S-equivalence classes of $\sigma$-semistable degree $d$, rank $r$, torsion-free sheaves on $C$. The automorphism group $\mathrm{Aut}(\zeta)$ of the stable map $\zeta$ acts on the moduli space $\overline{J}_{\zeta}(\sigma)$, with the action corresponding to taking pullbacks of sheaves.

\begin{prop}
	The fibre of the map $\overline{J}(\sigma) \to \overline{M}_{g,n}(X, \beta)$ over the closed point $[\zeta] \in \overline{M}_{g,n}(X, \beta)$ is isomorphic to the quotient $\overline{J}_{\zeta}(\sigma)/\mathrm{Aut}(\zeta)$.
\end{prop}

\begin{proof}
	Applying the final statement of Proposition \ref{prop: further quotients}, it is enough to show that the image of the isotropy subgroup $SL(W)_{j_0}$ of the point $j_0$ under the $SL(W)$ action on $I$ inside $PGL(W)$ can be identified with $\mathrm{Aut}(\zeta)$. But this essentially follows from the argument given within the proof of Proposition \ref{prop: quotient stack description 1} used to show that the morphism $\Phi$ is fully faithful (take the base $S$ to be $\spec \C$).
\end{proof}

This yields the following description of the closed points of $\overline{J}(\sigma)$, analogous to the description arising from \cite[Theorem 8.2.1]{pand}.

\begin{prop} \label{prop: description of the closed points}
	Let $\zeta$ be a stable map in $\overline{\M}_{g,n}(X, \beta)$, and let $[\zeta]$ be the corresponding point in the coarse moduli space $\overline{M}_{g,n}(X, \beta)$. Denote the underlying nodal curve of $\zeta$ by $C$. Then the closed points of $\overline{J}(\sigma)$ lying over the closed point $[\zeta]$ are in 1-1 correspondence with equivalence classes of $\sigma$-semistable rank $r$, degree $d$, torsion free sheaves $E$ over $C$, where sheaves $E$ and $E'$ are equivalent if and only if the Jordan-Hölder factors of these sheaves differ by an automorphism of the stable map $\zeta$. \qed
\end{prop}

\begin{remark}
	It follows from Proposition \ref{prop: description of the closed points} that the open substack $\overline{\J}^{s}(\sigma) \subset \overline{\J}(\sigma)$ parametrising $\sigma$-stable sheaves is saturated with respect to the good moduli space morphism $\pi_{\sigma} : \overline{\J}(\sigma) \to \overline{J}(\sigma)$, hence $\overline{J}^{s}(\sigma) := \pi_{\sigma}(\overline{\J}^{s}(\sigma))$ is an open subscheme of $\overline{J}(\sigma)$ which is a coarse moduli space for $\overline{\J}^{s}(\sigma)$.
\end{remark}

\section{Wall and Chamber Decompositions} \label{section: wall chamber decompositions}

We now turn to proving Theorem \ref{thm: Theorem B}, which concerns what happens as the parameter $\sigma$ varies in the space of stability conditions.

\subsection{Existence of Stability Decompositions} Here we closely follow \cite[Section 4]{grebrosstoma}. We keep the notation of Section \ref{section: carrying out the construction}. Fix $\pi_{\UU}$-ample $\Q$-invertible sheaves $\LL_1, \dots, \LL_k$ and let $\Sigma' = \{ \sigma \in (\Q^{\geq 0})^k \setminus \{0\} : \sum \sigma_i = 1 \}$ denote the resulting space of stability conditions.

\begin{defn}
	A \emph{chamber structure} on $\Sigma'$ is given by a collection $\{W_i : i \in I\}$ of \emph{walls}, where each wall is either a real hypersurface in $\Sigma'_{\R} = \{ \sigma \in (\R^{\geq 0})^k \setminus \{0\} : \sum \sigma_i = 1 \}$ (the \emph{proper walls}), or all of $\Sigma'_{\R}$. A subset $\mathcal{C} \subset \Sigma'_{\R}$ which is maximally connected with respect to the property that for each $i \in I$, either $\mathcal{C} \subset W_i$ or $\mathcal{C} \cap W_i = \emptyset$, is called a \emph{chamber}. The chamber structure is said to be \emph{rational linear} if each proper wall $W_j$ is a rational hyperplane, and \emph{proper} if all walls are proper.
\end{defn}

Consider the collection of sheaves $\mathcal{S}$ consisting of all sheaves $F$ for which there exists $\sigma \in \Sigma'$ and a $\sigma$-semistable, degree $d$, rank $r$, torsion-free coherent sheaf $E$ over some stable map $\zeta = (C; x_1, \dots, x_n; f : C \to X)$ in $\overline{\M}_{g,n}(X, \beta)$, such that $F$ is a saturated subsheaf of $E$ and such that for some $\sigma' \in \Sigma'$ (not necessarily equal to $\sigma$), one has $\mu_{\sigma'}(F) \geq \mu_{\sigma'}(E)$.

\begin{lemma}[cf. \cite{grebrosstoma}, Lemma 4.5]
	The collection $\mathcal{S}$ is a bounded family of sheaves.
\end{lemma}

\begin{proof}
	Suppose $E \supset F \in \mathcal{S}$ with $E$ $\sigma$-semistable and $\mu^{\sigma'}(F) \geq \mu^{\sigma'}(E)$. By Lemma 2.11 of \emph{loc. cit.} there exists $j \in \{1, \dots, k\}$ with $\mu^{\LL_j}(F) \geq \mu^{\sigma'}(E)$. Since by Lemma \ref{lem: boundedness result} $E$ varies in a bounded family, and since the quotient $E/F$ is torsion-free or zero, the boundedness of $\mathcal{S}$ follows by applying Grothendieck's lemma \cite[Lemma 1.7.9]{huybrechts_lehn}, as in the proof of \cite[Lemma 4.5]{grebrosstoma}.
\end{proof}

Now suppose $E \supset F \in \mathcal{S}$ is a sheaf over some stable map $\zeta$ in $\overline{\M}_{g,n}(X, \beta)$. Suppose the underlying curve $C$ of $\zeta$ has irreducible components $C_1, \dots, C_{\rho}$. Set
$$ W_{F, \zeta} := \left\{ \sigma \in \Sigma'_{\R} : \sum_{i=1}^k \sigma_i \left( r\chi(F) \sum_{j=1}^{\rho}  \deg_{C_j} L_i - \chi(E) \sum_{j=1}^{\rho} r_j(F) \deg_{C_j} L_i \right) = 0 \right\}. $$
Each $W_{F, \zeta}$ is empty, the whole of $\Sigma'$ or a rational hyperplane in $\Sigma'_{\R}$. In the first case the wall $W_{F, \zeta}$ is discarded. As the wall $W_{F,\zeta}$ depends only on the Euler characteristic and the multirank of $F$, and as $\mathcal{S}$ is bounded, letting $F \in \mathcal{S}$ vary yields finitely many possibilities for the proper walls which arise in this way, each of which is a rational hyperplane. This gives rise to a finite rational linear chamber structure on the set $\Sigma'$.

\begin{lemma} \label{lem: stability inside a chamber}
	Suppose $E$ is a degree $d$, rank $r$, torsion-free coherent sheaf over some stable map $\zeta$ which is semistable with respect to some $\sigma \in \Sigma'$, and suppose $F$ is a saturated subsheaf of $E$ with $F \in \mathcal{S}$. Suppose additionally that $\sigma', \sigma'' \in \Sigma'$ lie in a common chamber $\mathcal{C} \subset \Sigma'$. Then $\mu_{\sigma'}(F) < (\leq) \ \mu_{\sigma'}(E)$ if and only if $\mu_{\sigma''}(F) < (\leq) \ \mu_{\sigma''}(E)$.
\end{lemma}

\begin{proof}
	The proof is very similar to that of \cite[Lemma 4.6]{grebrosstoma}. We deal with the non-strict inequality case; the strict inequality case follows by essentially the same argument. Swapping $\sigma'$ and $\sigma''$ if necessary, suppose for a contradiction that
	$$ \mu_{\sigma'}(F) \leq \mu_{\sigma'}(E) \quad \mathrm{and} \quad \mu_{\sigma''}(F) > \mu_{\sigma''}(E). $$
	Consider the linear function $h : \Sigma'_{\R} \to \R$ given by
	$$ h(\sigma) = \sum_{i=1}^k \sigma_i \left( r\chi(F) \sum_{j=1}^{\rho}  \deg_{C_j} L_i - \chi(E) \sum_{j=1}^{\rho} r_j(F) \deg_{C_j} L_i \right). $$
	By definition $W_{F,\zeta}$ is the zero set of $h$. By hypothesis, we have $h(\sigma') \leq 0$ and $h(\sigma'') > 0$. Since $\sigma'$ and $\sigma''$ are both contained in the chamber $\mathcal{C}$, it follows that $\mathcal{C} \cap W_{F,\zeta} = \emptyset$. On the other hand there exists $\sigma''' \in \Sigma'$ on the line segment joining $\sigma'$ and $\sigma''$ with $h(\sigma''') = 0$. As the chamber $\mathcal{C}$ is convex, $\mathcal{C}$ must contain this line segment, which implies $\mathcal{C} \cap W_{F, \zeta} \not = \emptyset$, a contradiction.
\end{proof}

\begin{prop} \label{prop: wall chamber structure}
	There exists a rational linear chamber structure on $\Sigma'$ cut out by finitely many walls, such that if $\sigma', \sigma'' \in \Sigma'$ belong to the same chamber, then:
	\begin{enumerate}
		\item if $E$ is a degree $d$, rank $r$, torsion-free coherent sheaf over a stable map in $\overline{\M}_{g,n}(X,\beta)$, then $E$ is $\sigma'$-semistable if and only if $E$ is $\sigma''$-semistable; and
		\item if $E$ and $E'$ are degree $d$, rank $r$, torsion-free coherent sheaves over the same stable map in $\overline{\M}_{g,n}(X,\beta)$ which are both semistable with respect to both $\sigma'$ and $\sigma''$, then $E$ and $E'$ are S-equivalent with respect to $\sigma'$ if and only if they are S-equivalent with respect to $\sigma''$.
	\end{enumerate}
	Moreover, if $\sigma \in \Sigma'$ does not lie in any wall then all $\sigma$-semistable sheaves are $\sigma$-stable.
\end{prop}

\begin{defn}
	 We refer to this chamber structure on $\Sigma'$ as the \emph{stability chamber decomposition} (with respect to $\LL_1, \dots, \LL_k$), and to the corresponding chambers as \emph{stability chambers}.
\end{defn}

\begin{proof}[Proof of Proposition \ref{prop: wall chamber structure}.]
	The proof of the first two assertions proceeds along very similar lines to the proof of Proposition 4.2 of \emph{loc. cit.} Suppose $\sigma', \sigma'' \in \Sigma'$ lie in the same chamber. Let $E$ be a $\sigma'$-semistable, degree $d$, uniform rank $r$, torsion-free coherent sheaf over a stable map in $\overline{\mathcal{M}}_{g,n}(X, \beta)$, and let $F \subset E$ be a saturated subsheaf. If $\mu^{\sigma''}(F) < \mu^{\sigma''}(E)$ then $F$ does not destabilise $E$; otherwise $E \supset F \in \mathcal{S}$ and so $\mu^{\sigma''}(F) \leq \mu^{\sigma''}(E)$ by Lemma \ref{lem: stability inside a chamber}. As such, $E$ is also $\sigma''$-semistable; this establishes the first assertion.
		
	To establish the second assertion, suppose $E$ is strictly $\sigma'$-semistable, so that there exists a saturated subsheaf $F \subset E$ with $\mu^{\sigma'}(F) = \mu^{\sigma'}(E)$. By Lemma \ref{lem: stability inside a chamber}, for any such subsheaf $F$ we have $\mu^{\sigma''}(F) = \mu^{\sigma''}(E)$. In particular, a maximal length Jordan--H\"older filtration of $E$ with respect to $\sigma'$ must also be a maximal length Jordan--H\"older filtration of $E$ with respect to $\sigma''$ and vice versa, so $\mathrm{gr}_{\sigma'}(E) \cong \mathrm{gr}_{\sigma''}(E)$; this proves the second assertion.
	
	It remains to show that if $\sigma \in \Sigma'$ is not contained in any wall $W_{F,\zeta}$ then all $\sigma$-semistable sheaves are $\sigma$-stable. However, if $E$ is a $\sigma$-semistable degree $d$, uniform rank $r$ torsion-free coherent sheaf over the stable map $\zeta$ then for all saturated subsheaves $F \subset E$ we must have $\mu^{\sigma}(F) < \mu^{\sigma}(E)$, which implies that $E$ is $\sigma$-stable.
\end{proof}

\subsection{The Mumford-Thaddeus Principle} As a consequence of the master space construction of Section \ref{section: carrying out the construction} and the existence of a finite stability decomposition on the space of stability conditions $\Sigma'$, the moduli spaces $\overline{J}(\sigma)$ are always related by finite sequences of Thaddeus flips (cf. Definition \ref{defn: thaddeus flip}).

\begin{theorem} \label{thm: mumford thaddeus principle}
	With the notation as in Section \ref{section: GIT construction setup}, let $\sigma_0, \sigma_1 \in \Sigma'$ be stability parameters defined with respect to the relatively ample $\Q$-invertible sheaves $\LL_1, \dots, \LL_k$. Let $\sigma_t = (1-t)\sigma_0 + t\sigma_1$ for $t \in \Q \cap [0,1]$. Then the good moduli spaces $\overline{J}(\sigma_0)$ and $\overline{J}(\sigma_1)$ are related by a finite number of Thaddeus flips of the form
	\begin{equation} \label{eqn: thaddeus flip crossing a wall}
		\begin{tikzcd}
			{\overline{J}(\sigma_{t_i})} &&&& {\overline{J}(\sigma_{t_{i+1}})} \\
			\\
			&& {\overline{J}(\sigma_{t_i'})}
			\arrow[from=1-1, to=3-3]
			\arrow[from=1-5, to=3-3]
		\end{tikzcd}
	\end{equation}
	for some $t_i, t_i' \in \Q \cap [0,1]$. In particular, $\overline{J}(\sigma_0)$ and $\overline{J}(\sigma_1)$ are related by a finite number of Thaddeus flips through moduli spaces of rank $r$, degree $d$, torsion-free coherent sheaves over stable maps in $\overline{\M}_{g,n}(X, \beta)$.
\end{theorem}

\begin{proof}
	The proof proceeds along similar lines to that of \cite[Theorem 5.2]{grebrosstomamaster}. If $\sigma_0$ and $\sigma_1$ define the same stability condition then there is nothing to prove. Suppose instead that $\sigma_0$ and $\sigma_1$ lie in different stability chambers. Without loss of generality, we may assume that the $\LL_i$ are genuine relatively very ample invertible sheaves. Let $\mathfrak{S} \subset \Sigma' \cap \{\sigma_t : t \in \Q \cap [0,1]\}$ be a set of representative stability conditions for each of the stability chambers the line segment $\sigma_t$ intersects, with exactly one representative in $\mathfrak{S}$ for each such chamber; for the chambers containing $\sigma_0$ and $\sigma_1$, we choose these stability conditions as representatives. As the stability decomposition is cut out by finitely many walls, the set $\mathfrak{S}$ is finite.
	
	As in Section \ref{section: GIT construction setup} form the master space $Z = Z_{\mathfrak{S}}$; the moduli spaces $\overline{J}(\sigma_t)$ obtained as $t \in [0,1]$ varies are all GIT quotients of this master space. Every time a wall $W_i$ separating chambers $\mathcal{C}_i$ and $\mathcal{C}_{i+1}$ with respective representative stability conditions $\sigma_{t_i}$ and $\sigma_{t_{i+1}}$ is crossed by $\sigma_t$, taking limits in the inequalities \eqref{eqn: multi slope stability nodal curve} give inclusions
	$$ Z_r^{\sigma_{t_i}-ss} \subset Z_r^{\sigma_{t_i'}-ss} \supset Z_r^{\sigma_{t_{i+1}}-ss}. $$
Here $\sigma_{t_i'}$ is the stability condition between $\sigma_{t_i}$ and $\sigma_{t_{i+1}}$ lying on the wall $W_i$. The moduli spaces $\overline{J}(\sigma_{t_i})$ and $\overline{J}(\sigma_{t_{i+1}})$ are then related by the Thaddeus flip\footnote{This is a Thaddeus flip in the sense of Definition \ref{defn: thaddeus flip}, since the above loci are relative GIT semistable loci, cf. Theorem \ref{thm: good moduli spaces as GIT quotients}.} \eqref{eqn: thaddeus flip crossing a wall} through the moduli space $\overline{J}(\sigma_{t_i'})$. As there are finitely many wall crossings between $\sigma_0$ and $\sigma_1$, we are done.
\end{proof}

\bibliographystyle{acm}
\bibliography{cujrefs.bib}	

\begin{thebibliography}{10}

\bibitem{abramovichcortivistoli}
{\sc Abramovich, D., Corti, A., and Vistoli, A.}
\newblock {Twisted Bundles and Admissible Covers}.
\newblock {\em Comm. Alg. 31}, 8 (2003), 3547--3618.

\bibitem{abreupacinitrop}
{\sc Abreu, A., and Pacini, M.}
\newblock {The Universal Tropical Jacobian and the Skeleton of the Esteves' Universal Jacobian}.
\newblock {\em Proc. Lond. Math. Soc. 120}, 3 (2020), 328--369.

\bibitem{abreupacini}
{\sc Abreu, A., and Pacini, M.}
\newblock {The Resolution of the Universal Abel Map via Tropical Geometry and Applications}.
\newblock {\em Adv. Math. 378\/} (2021), 107520.

\bibitem{abreupagani}
{\sc Abreu, A., and Pagani, N.}
\newblock {Wall-Crossing of Universal Brill-Noether Classes}.
\newblock {\em arXiv preprint arXiv:2303.16836\/} (2023).

\bibitem{alexeev}
{\sc Alexeev, V.}
\newblock {Compactified Jacobians and Torelli Map}.
\newblock {\em Pub. Res. Inst. Math. 40}, 4 (2004), 1241--1265.

\bibitem{alpergood}
{\sc Alper, J.}
\newblock {Good Moduli Spaces for Artin Stacks}.
\newblock {\em Ann. Inst. Fourier 63}, 6 (2013), 2349--2402.

\bibitem{altmankleiman}
{\sc Altman, A., and Kleiman, S.}
\newblock {Compactifying the Picard Scheme}.
\newblock {\em Adv. Math. 35}, 1 (1980), 50--112.

\bibitem{alvarezconsulking}
{\sc {\'A}lvarez-C{\'o}nsul, L., and King, A.}
\newblock {A Functorial Construction of Moduli of Sheaves}.
\newblock {\em Inv. Math. 168}, 3 (2007), 613--666.

\bibitem{baldwinpositive}
{\sc Baldwin, E.}
\newblock {A GIT Construction of Moduli Spaces of Stable Maps in Positive Characteristic}.
\newblock {\em J. Lond. Math. Soc. 78}, 1 (2008), 107--124.

\bibitem{baldwinswinarski}
{\sc Baldwin, E., and Swinarski, D.}
\newblock {A Geometric Invariant Theory Construction of Moduli Spaces of Stable Maps}.
\newblock {\em Int. Math. Res. Not. 2008\/} (2008).

\bibitem{behrendgw}
{\sc Behrend, K.}
\newblock {Gromov-Witten Invariants in Algebraic Geometry}.
\newblock {\em Inv. Math. 127}, 3 (1997), 601--617.

\bibitem{behrendmanin}
{\sc Behrend, K., and Manin, Y.}
\newblock {Stacks of Stable Maps and Gromov-Witten Invariants}.
\newblock {\em Duke Math. J. 85}, 1 (1996), 1--60.

\bibitem{bfmv}
{\sc Bini, G., Felici, F., Melo, M., and Viviani, F.}
\newblock {\em {Geometric Invariant Theory for Polarized Curves}}.
\newblock Lecture Notes in Mathematics 2122. Springer, 2014.

\bibitem{caporaso}
{\sc Caporaso, L.}
\newblock {A Compactification of the Universal Picard Variety over the Moduli Space of Stable Curves}.
\newblock {\em J. Am. Math. Soc. 7}, 3 (1994), 589--660.

\bibitem{caporasoneron}
{\sc Caporaso, L.}
\newblock {N{\'e}ron Models and Compactified Picard Schemes over the Moduli Stack of Stable Curves}.
\newblock {\em Am. J. Math. 130}, 1 (2008), 1--47.

\bibitem{caporasochrist}
{\sc Caporaso, L., and Christ, K.}
\newblock {Combinatorics of Compactified Universal Jacobians}.
\newblock {\em Adv. Math. 346\/} (2019), 1091--1136.

\bibitem{casalainakassvivianilocal}
{\sc Casalaina-Martin, S., Kass, J.~L., and Viviani, F.}
\newblock {The Local Structure of Compactified Jacobians}.
\newblock {\em Proc. Lond. Math. Soc. 110}, 2 (2014), 510--542.

\bibitem{casalainakassviviani}
{\sc Casalaina-Martin, S., Kass, J.~L., and Viviani, F.}
\newblock {The Singularities and Birational Geometry of the Compactified Universal Jacobian}.
\newblock {\em J. Algebr. Geom. 4}, 3 (2017), 353--393.

\bibitem{conradkm}
{\sc Conrad, B.}
\newblock {Keel--Mori Theorem via Stacks}, 2005.

\bibitem{dolgachevhu}
{\sc Dolgachev, I., and Hu, Y.}
\newblock {Variation of Geometric Invariant Theory Quotients}.
\newblock {\em Pub. Math. IHÉS 87}, 1 (1998), 5--51.

\bibitem{dudin}
{\sc Dudin, B.}
\newblock {Compactified Universal Jacobian and the Double Ramification Cycle}.
\newblock {\em Int. Math. Res. Not. 2018}, 8 (2018), 2416--2446.

\bibitem{estevesreljacobian}
{\sc Esteves, E.}
\newblock {Compactifying the Relative Jacobian over Families of Reduced Curves}.
\newblock {\em Trans. Am. Math. Soc. 353}, 8 (2001), 3045--3095.

\bibitem{fringuelli}
{\sc Fringuelli, R.}
\newblock {The Picard Group of the Universal Moduli Space of Vector Bundles on Stable Curves}.
\newblock {\em Adv. Math. 336\/} (2018), 477--557.

\bibitem{fultonpandharipande}
{\sc Fulton, W., and Pandharipande, R.}
\newblock {Notes on Stable Maps and Quantum Cohomology}.
\newblock {\em arXiv preprint alg-geom/9608011\/} (1996).

\bibitem{gieseker}
{\sc Gieseker, D.}
\newblock {\em {Lectures on Moduli of Curves}}.
\newblock Tata Institute Lectures on Mathematics and Physics. Springer-Verlag, 1982.

\bibitem{grebrosstoma}
{\sc Greb, D., Ross, J., and Toma, M.}
\newblock {Variation of Gieseker moduli spaces via quiver GIT}.
\newblock {\em Geometry \& Topology 20}, 3 (2016), 1539--1610.

\bibitem{grebrosstomamaster}
{\sc Greb, D., Ross, J., and Toma, M.}
\newblock {A Master Space for Moduli Spaces of Gieseker-Stable Sheaves}.
\newblock {\em Transformation Groups 24}, 2 (2019), 379--401.

\bibitem{ega3_2}
{\sc Grothendieck, A.}
\newblock {\'{E}l\'ements de G\'eom\'etrie Alg\'ebrique III. \'{E}tude Cohomologique des Faisceaux Coh\'erents, Seconde Partie}.
\newblock {\em Pub. Math. IHÉS 17\/} (1963), 5--91.

\bibitem{hall}
{\sc Hall, J.}
\newblock {Moduli of Singular Curves}.
\newblock {\em arXiv preprint arXiv:1011.6007\/} (2010).

\bibitem{holmeskasspagani}
{\sc Holmes, D., Kass, J.~L., and Pagani, N.}
\newblock {Extending the Double Ramification Cycle using Jacobians}.
\newblock {\em Eur. J. Math. 4}, 3 (2018), 1087--1099.

\bibitem{hurelative}
{\sc Hu, Y.}
\newblock {Relative Geometric Invariant Theory and Universal Moduli Spaces}.
\newblock {\em Int. J. Math. 7}, 2 (1996), 151--181.

\bibitem{huybrechts_lehn}
{\sc Huybrechts, D., and Lehn, M.}
\newblock {\em {The Geometry of Moduli Spaces of Sheaves}}, 2~ed.
\newblock Cambridge Mathematical Library. Cambridge University Press, 2010.

\bibitem{kptheta}
{\sc Kass, J.~L., and Pagani, N.}
\newblock {Extensions of the Universal Theta Divisor}.
\newblock {\em Adv. Math. 321\/} (Dec 2017), 221--268.

\bibitem{kasspagani}
{\sc Kass, J.~L., and Pagani, N.}
\newblock {The Stability Space of Compactified Universal Jacobians}.
\newblock {\em Trans. Am. Math. Soc. 372}, 7 (June 2019), 4851--4887.

\bibitem{kingquiverreps}
{\sc King, A.}
\newblock {Moduli of Representations of Finite Dimensional Algebras}.
\newblock {\em Quart. J. Math. 45}, 4 (1994), 515--530.

\bibitem{kollar}
{\sc Koll{\'a}r, J.}
\newblock {Projectivity of Complete Moduli}.
\newblock {\em J. Diff. Geom. 32}, 1 (July 1990), 235--268.

\bibitem{melopicard}
{\sc Melo, M.}
\newblock {Compactified Picard Stacks over $\overline{\mathcal{M}}_g$}.
\newblock {\em Math. Z. 263}, 4 (2009), 939--957.

\bibitem{melopicardmarked}
{\sc Melo, M.}
\newblock {Compactified Picard Stacks over the Moduli Stack of Stable Curves with Marked Points}.
\newblock {\em Adv. Math. 226}, 1 (2011), 727--763.

\bibitem{melocuj}
{\sc Melo, M.}
\newblock {Compactifications of the Universal Jacobian over Curves with Marked Points}.
\newblock {\em arXiv preprint arXiv:1509.06177\/} (2015).

\bibitem{tropunivjacobian}
{\sc Melo, M., Molcho, S., Ulirsch, M., and Viviani, F.}
\newblock {Tropicalization of the Universal Jacobian}.
\newblock {\em Epijournal Geo. Alg. 6}, 15 (2022).

\bibitem{meloviviani}
{\sc Melo, M., and Viviani, F.}
\newblock {The Picard Group of the Compactified Universal Jacobian}.
\newblock {\em Doc. Math. 19\/} (2014), 457--507.

\bibitem{mukaimoduli}
{\sc Mukai, S.}
\newblock {\em {An Introduction to Invariants and Moduli}}, vol.~81 of {\em Cambridge Studies in Advanced Mathematics}.
\newblock Cambridge University Press, 2003.

\bibitem{mfk}
{\sc Mumford, D., Fogarty, J., and Kirwan, F.}
\newblock {\em {Geometric Invariant Theory}}, {Third}~ed.
\newblock Ergebnisse der Mathematik und Ihrer Grenzgebiete. Springer-Verlag, 1994.

\bibitem{odaseshadri}
{\sc Oda, T., and Seshadri, C.}
\newblock {Compactifications of the Generalized Jacobian Variety}.
\newblock {\em Trans. Am. Math. Soc. 253\/} (1979), 1--90.

\bibitem{paganitommasi}
{\sc Pagani, N., and Tommasi, O.}
\newblock {Geometry of Genus One Fine Compactified Universal Jacobians}.
\newblock {\em Int. Math. Res. Not. 0\/} (04 2022).

\bibitem{pand}
{\sc Pandharipande, R.}
\newblock {A Compactification over $\overline{M_g}$ of the Universal Moduli Space of Slope-Semistable Vector Bundles}.
\newblock {\em J. Am. Math. Soc. 9}, 2 (1996), 425--471.

\bibitem{schmitthilbert}
{\sc Schmitt, A.}
\newblock {The Hilbert Compactification of the Universal Moduli Space of Semistable Vector Bundles over Smooth Curves}.
\newblock {\em J. Diff. Geom. 66}, 2 (2004), 169--209.

\bibitem{schmittrgitremark}
{\sc Schmitt, A.}
\newblock {A Remark on Relative Geometric Invariant Theory for Quasi-Projective Varieties}.
\newblock {\em Math. Nachr. 292}, 2 (2019), 428--435.

\bibitem{seshadri}
{\sc Seshadri, C.~S.}
\newblock {\em {Fibr\'es Vectoriels sur les Courbes Alg\'ebriques}}.
\newblock No.~96 in Ast\'erisque. Soci\'et\'e Math\'ematique de France, 1982.

\bibitem{simpson}
{\sc Simpson, C.}
\newblock {Moduli of Representations of the Fundamental Group of a Smooth Projective Variety I}.
\newblock {\em Pub. Math. IHÉS 79\/} (1994), 47--129.

\bibitem{teixidor}
{\sc Teixidor~i Bigas, M.}
\newblock {Compactifications of Moduli Spaces of (Semi) Stable Bundles on Singular Curves: Two Points of View}.
\newblock {\em Coll. Math. 49}, 2 (1998), 527--548.

\bibitem{thaddeus}
{\sc Thaddeus, M.}
\newblock {Geometric Invariant Theory and Flips}.
\newblock {\em J. Am. Math. Soc. 9}, 3 (1996), 691--723.

\bibitem{stacks-project}
{\sc {The Stacks Project Authors}}.
\newblock {Stacks Project}.

\end{thebibliography}

\end{document}